\documentclass[11pt]{amsart}

\usepackage{xcolor}
\usepackage{amsmath, amsthm, amssymb, amsfonts, mathtools, enumerate}
\usepackage{hyperref}
\usepackage{verbatim} 
\usepackage{graphicx,epsfig,color}
\usepackage{exscale}

\theoremstyle{plain}
\DeclareMathOperator{\Int}{Int}

\newcommand{\cN}{\lceil N\rceil}

\newtheorem{theorem}{Theorem}[section]
\newtheorem*{th:re}{Theorem \ref{th:re}}
\newtheorem{fact}[theorem]{Fact}
\newtheorem{lemma}[theorem]{Lemma}

\newtheorem{corollary}[theorem]{Corollary}

\newtheorem{proposition}[theorem]{Proposition}

\theoremstyle{definition}

\newtheorem{definition}[theorem]{Definition}

\newtheorem{assumption}[theorem]{Assumption}

\theoremstyle{remark}
\newtheorem{remark}[theorem]{Remark}

\newtheorem{example}[theorem]{Example}

\newtheorem*{jjj}{Remark}

\newcommand{\red}[1]{{\color{black} #1}}

\DeclareSymbolFont{AMSb}{U}{msb}{m}{n}
\DeclareMathSymbol{\N}{\mathalpha}{AMSb}{"4E}
\DeclareMathSymbol{\R}{\mathalpha}{AMSb}{"52}
\DeclareMathSymbol{\Z}{\mathalpha}{AMSb}{"5A}
\DeclareMathSymbol{\D}{\mathalpha}{AMSb}{"44}
\DeclareMathSymbol{\s}{\mathalpha}{AMSb}{"53}

\DeclareMathOperator{\md}{md}
\newcommand{\uN}{\lceil N\rceil}

\renewcommand{\Im}{\mbox{Im}}

\newcommand{\sF}{\scriptscriptstyle{F}}

\newcommand{\sM}{\scriptscriptstyle{M}}

\DeclareMathOperator{\tr}{tr}

\DeclareMathOperator{\vol}{vol}

\DeclareMathOperator{\supp}{supp}

\DeclareMathOperator{\de}{d}
\DeclareMathOperator{\m}{m}
\DeclareMathOperator{\ric}{Ric}

\newcommand{\T}{\mathcal{T}}

\def\eps{\epsilon}

\renewcommand{\L}{\mbox{L}}

\title[Glued spaces]{Glued spaces and lower Ricci curvature bounds}
\author{Christian Ketterer}
\thanks{ }
\thanks{{\it 2010 Mathmatics Subject Classification.} Primary 53C21, 54E35. Keywords: metric measure space,  Ricci curvature, curvature-dimension condition, gluing construction.}
\address{Logic House, Department of Mathematics and Statistics, South Campus, Maynooth University}
\email{christian.ketterer@mu.ie}
\begin{document}
\begin{abstract} We consider  Riemannian manifolds $M_i$, ${i=0,1}$, with boundary and  $\Phi_i\in C^{\infty}(M_i)$ non-negative such that  $(M_i, \Phi_i)$ has Bakry-Emery $N$-Ricci curvature bounded from below by $K$. Let $Y_0$ and $Y_1$ be isometric,  compact components of the boundary  of $M_0$ and $M_1$ respectively and assume $\Phi_0=\Phi_1$ on $Y_0\simeq Y_1$.  We assume that $\Pi_0+\Pi_1=:\Pi \geq 0$ (*), and $d\Phi_0(\nu_0)+ d\Phi_1(\nu_1)\leq \tr \Pi$ on $Y_0\simeq Y_1$ (**) where $\Pi_i$ is the second fundamental form and $\nu_i$ is inner unit normal field along $\partial M_i$. 
We show that the metric glued space $M=M_0\cup_{\mathcal I}M_1$ together with the measure $\Phi d\mathcal H^n$ satisfies the curvature-dimension condition $CD(K, N)$ where $\Phi: M\rightarrow [0,\infty)$ arises  tautologically from $\Phi_1$ and $\Phi_2$. Moreover, $(M, \Phi d\mathcal H^n)$ is the collapsed Gromov-Hausdorff limit of smooth, $\lceil N \rceil$-dimensional Riemannian manifolds with Ricci curvature bounded from below by $K- \epsilon$ and is also the measured Gromov-Hausdorff limit of smooth, weighted  Riemannian manifolds such that the Bakry-Emery $\lceil N \rceil$-Ricci curvature is bounded from below by $K-\epsilon$. 
On the other hand we show that given a glued manifold as described it satisfies the curvature-dimension condition $CD(K,N)$ only if the condition (*) and (**) hold. The latter statement generalizes a theorem of Kosovski\u{\i} for sectional lower curvature bounds and  especially applies for the  case  $\Phi\equiv 1$ where a lower Ricci curvature bound and $\dim_{M_i}\leq N$ replaces a lower Bakry-Emery $N$-Ricci curvature bound.
\end{abstract}
\maketitle
\tableofcontents
\section{Introduction and Statement of Main Results}
In the context of lower curvature bounds  Petrunin proved  that the gluing  along the isometric boundary of two finite dimensional Alexandrov  spaces with curvature bounded from below by   $k\in \R$ is again an Alexandrov space with curvature bounded from below by   $k$ \cite{petruningluing}.  The gluing theorem for 2-dimensional Alexandrov spaces is due to A. D. Alexandrov. 

For Riemannian manifolds Kosovski\u{\i} proved the following generalization: Given two Riemannian manifolds with boundary and sectional curvature bounded from below  by $k$ such that there exists an isometry between their boundaries,   the sum of the corresponding second fundamental forms is pointwise positive semi-definite
 if and only if the glued space is an Alexandrov space with curvature bounded from below $k$ \cite{kosovskiigluing, kosovski}. More precisely Kosovski\u{\i} proved that 
 the glued space is the uniform limit of a sequence of Riemannian manifolds with sectional curvature bounded from below by $K-\epsilon$ for $\epsilon>0$ arbitrarily small. A theorem that replaces lower sectional curvature bounds with Ricci curvature in the latter statement was   established before by Perelman \cite{perelmanlarge}. Perelman used this to show the existence of Riemannian manifolds with volume bounded below,  positive Ricci curvature and arbitrarily large Betti numbers. 
Other applications  of gluing constructions  for instance  appear in  \cite{reisersurgery, reisersurgery2}, \cite{burdick, burdickthesis}, \cite{bww}, \cite{bns23} and in \cite{wong}.

In \cite{sch} Schlichting uses Kosovski\u{\i}'s method to prove a gluing theorem  for other curvature conditions including  positive isotropic curvature and lower bounds for the Riemannian curvature operator. Nonnegativity of the sum of the second fundamental forms along the isometric boundary is again  crucial.  For lower scalar curvature bounds  the   condition on the boundary is a lower bound for the sum of the mean curvatures \cite{miaolocalized, sch, gromovlawson} (see also \cite{sch} and \cite{bh}). A generalization of Petrunin's theorem for synthetic Ricci curvature bounds in the context of Alexandrov metric spaces was proven by the author together with V. Kapovitch and K.-T. Sturm \cite{kakest}.

These results and their applications show why gluing constructions, and  how they preserve lower curvature bounds, are  interesting and relevant  problems in smooth and nonsmooth geometry. In this article we address the gluing problem from the viewpoint of the  curvature-dimension condition in the sense of Lott-Sturm-Villani in the setting of weighted Riemannian manifolds. In our main theorem we will generalize Kosovski\u{\i}'s characterization  result of sectional curvature lower bounds for glued spaces to the setting of Ricci and also Bakry-Emery $N$-Ricci curvature lower bounds.

Let $(M_i,  g_i, \Phi_i)$, $i=0,1$, be \red{complete}, weighted Riemannian manifolds with boundary $\partial M_i$ where $\Phi_i\in C^\infty(M_i)$ with $\Phi_i\geq 0$. \red{We assume the sectional curvature is bounded from above and below by constants $\underline \kappa$ and $\overline \kappa$, respectively, and $\sup_{M_i} |\nabla \Phi_i|<\infty$. This holds, for instance, if $M_i$ is compact.}
We set   $ \mathring M_i:= M_i\backslash \Phi_i^{-1}(\{0\})$. Let $p\in \mathring M_i$ and let $v\in T_pM$. For a  constant $N>\dim_{M_i}=:n$ the Bakry-Emery $N$-Ricci tensor of $(M_i, g_i, \Phi_i)$ in $p$ is defined as
\begin{align*}
\ric^{\Phi_i, N}_{g_{i}}|_p(v,v)
= \ric_{g_i}|_p(v,v) - (N-n) \frac{ \nabla^2 \Phi_i^{\frac{1}{N-n}}|_p(v,v)}{\Phi_i^{\frac{1}{N-n}}(p)}.
\end{align*}
We write $\ric_{g_i}^{\Phi_i, N}\geq K$ if $\ric^{\Phi_i, N}_{g_i}|_p\geq Kg_i|_p$ for $p\in \mathring M_i$. If $\Phi_i\equiv const$ and $N=n$, $\ric_{g_i}^{\Phi_i, N}$ is  defined as the standard Ricci tensor (Definition \ref{def:baem}).

Assume $X_i=\Phi_i^{-1}(\{0\})\cap \partial M_i $ is empty or a closed and connected component of $\partial M_i$ and let $Y_i\neq\emptyset$ be another {\color{black} compact} and connected component of $\partial M_i$ with $X_i \cap Y_i = \emptyset$. 
We assume that there exists a Riemannian isometry $\mathcal I:   Y_0\rightarrow  Y_1$. In this case one can define the glued space $M_0\cup_{\mathcal I} M_1$ between $M_0$ and $M_1$ along $Y_0\simeq Y_1$.  $M_0\cup_{\mathcal I} M_1$ is diffeomorphic to a smooth manifold $M$ with boundary equipped with  the $C^0$ Riemannian metric $g$ defined via $g|_{M_i} = g_i$, $i=0,1$.  $d_g$ denotes the induced distance function (Section \ref{sec:gluing}).  We assume $\Phi_0(x)= \Phi_1(\mathcal I(x))>0$ for all $x\in Y_0$ and  define 
\begin{center}$
\Phi = \begin{cases} \Phi_0 & \mbox{ on } M_0\\
\Phi_1 &  \mbox{ on } M_1. 
\end{cases} 
$\end{center}
Let $\Pi_i$ be the second fundamental form of $\partial M_i$ and $\nu_i$ the inner unit normal vector field along $\partial M_i$.
We define $\lceil N \rceil:=\min\{n\in \N: n\geq N\}$. 
\begin{theorem}\label{main2}  Let $K\in \R$ and $N\in [1,\infty)$. 
 Assume for $i=0,1$ that $X_i=\emptyset$,  $\ric_{g_i}^{\Phi_i, N}\geq K$ and 
\begin{itemize}
\item[(1)] $\Pi_1 + \Pi_2 =: \Pi \geq 0$ on $Y_0\simeq Y_1$,  
\medskip
\item[(2)] $\langle N_0, \nabla \log \Phi_0\rangle + \langle N_1, \nabla \log \Phi_1\rangle\leq \tr \Pi$ on $Y_0\simeq Y_1$.
\end{itemize}
\smallskip
There exist $\lceil N \rceil$-dimensional Riemannian manifolds $(M_i, g_i)_{i\in \N}$  {\color{black}with boundary }such that $\ric_{g_i}\geq K-\epsilon(i)$ with $\epsilon(i)\downarrow 0$ for $i\rightarrow \infty$ converging in measured Gromov-Hausdorff sense to the metric measure space  $(M, d_g, \Phi \vol_g)$.
\end{theorem} 
We note that $\tr \Pi_i- g_i( N_i, \nabla \log \Phi_i)=: H^{\Phi_i}$ is the generalized mean curvature of the boundary of the weighted space $(M_i, g_i, \Phi_i)$ that was studied before, for instance, in \cite{milman, kettererhk, bukemcwo}.

\begin{theorem} \label{main3}  Let $K\in \R$ and $N\in [1,\infty)$. 
 Assume for $i=0,1$ that  $\ric_{g_i}^{\Phi_i, N}\geq K$ and 
\begin{itemize}
\item[(1)] $ \Pi \geq 0$ on $Y_0\simeq Y_1$,  
\medskip
\item[(2)] $\tr \Pi-\langle N_0, \nabla \log \Phi_0\rangle - \langle N_1, \nabla \log \Phi_1\rangle\geq 0 $ on $Y_0\simeq Y_1$.
\end{itemize}
\smallskip
Then  there exists a sequence of smooth Riemannian metrics $g^{n}$ on $M$ and functions $\Phi^{n}\in C^\infty(M)$ such that $\ric^{\lceil N \rceil, \Phi^{n}}_{g^n}\geq K-\epsilon(n)$ for $\epsilon(n)\downarrow 0$ as $n\rightarrow \infty$ and $g^n$ and $\Phi^n$ converge uniformly to $g$ and $\Phi$. 
\end{theorem} 
The {\it curvature-dimension condition} $CD(K,N)$ is a synthetic notion of Ricci curvature bounded from below and dimension bounded from above for metric measure spaces  (Definition \ref{def:cd}) that is stable under measured Gromov-Hausdorff convergence and equivalent to Bakry-Emery curvature bounded from below for weighted Riemannian manifolds with convex boundary \cite{stugeo2, lottvillani}.

We also obtain the following corollary.
\begin{corollary}\label{cor1} In addition to the previous properties assume that $\Pi_i\geq 0$ on $\partial M_i \backslash Y_i$. Then
the metric glued space $(M_0\cup_{\mathcal I}M_1, d_g)$  equpped with the measure $\Phi \vol_g$ satisfies the  condition $CD(K, N)$.
\end{corollary}

For the proof of Theorem \ref{main2} and Theorem \ref{main3} we apply an idea of Lott \cite{lobaem} (see also \cite{ketterer}): The problem  translates to the  setting of Riemannian manifolds by considering warped products of the form $M_i \times_{f_i}r\mathbb S^{\lceil N \rceil}$ where we choose $f_i=\Phi_i^{{1}/{(N-1)}}$.  Then we  apply the previously mentioned construction by Kosovski\u{\i} and Schlichting and  send the parameter $r$ to $0$. 
\begin{jjj}{In a previous version of this article we required additional conditions for the boundary components $X_i$ and $\partial M_i \backslash (Y_i \cup X_i)$. By localizing our arguments we could remove this assumptions.}
\end{jjj}

As already explained for an Alexandrov lower curvature bound  the condition on the second fundamental form is not only sufficient but  also necessary \cite{kosovskiigluing, kosovski}. We show that the conditions (1) and (2) are also necessary conditions for a synthetic lower Ricci curvature bound for the glued space. 
\begin{theorem}\label{th:re}
Assume the metric glued space $(M_0\cup_{\mathcal I} M_1, d_g)$ equipped with $\m= \Phi \vol_g$ satisfies a curvature-dimension condition $CD(K,N)$ for $K\in \R$ and $N\in [1, \infty)$. Then it follows
 $\Pi_i \geq 0$ on $\partial M_i \backslash Y_i$ and 
\smallskip
\begin{itemize}
\item[(1)] $\Pi_1 + \Pi_2  \geq 0$ on $Y_0\simeq Y_1$,  
\medskip
\item[(2)] $H^{\Phi_0}+ H^{\Phi_1}\geq  0$ on $Y_0\simeq Y_1$.
\end{itemize}
\end{theorem}
The proof of 
Theorem \ref{th:re} is based on a characterizaton of the curvature-dimension condition via $1D$ localisation w.r.t. $1$-Lipschitz functions.  We  note that,  by stability of the condition $CD(K,N)$, the existence of a sequence as in Theorem \ref{main2} or as in Theorem \ref{main3} also yields the conclusion of Theorem \ref{th:re}. Theorem \ref{th:re} applies in particular to the "noncollapsed" case, i.e. $\Phi_0\equiv \Phi_1\equiv 1$. 

We get the following Corollary. 
\begin{corollary}
Let $M^n_i$, $i=0,1$, be Riemannian manifolds with $\ric_{M_i}\geq K$ and compact boundary. Let $\mathcal I: \partial M_0\rightarrow \partial M_1$ be an isometry. Then the glued space $M_0\cup_{\mathcal I} M_1$ satisfies the curvature-dimension condition $CD(K,n)$ if and only if $\Pi_1 + \Pi_2\geq 0$ on $\partial M_0\simeq \partial M_1$.
\end{corollary}
\begin{remark}[Doubling]
Corollary \ref{cor1} applies when we glue together  two copies  of a weighted Riemannian manifold $(M,g_M,\Phi)$ that satisfies the curvature-dimension condition $CD(K,N)$ along a boundary component. For this one needs that  the weight $\Phi$ satisfies $\Phi|_Y>0$ and $\langle \nabla \Phi, \nu\rangle |_Y\leq 0$ where $\nu$ is the  inner unit normal vector field along $Y$. 
In \cite{giorgiorizzi} the authors observe that   a doubling construction for a  metric measure space does not preserve the Riemannian curvature-dimension condition in general. Their example  is the Grushin half plane  $\mathbb G$ equipped with a suitable weight.  $\mathbb G$ is the closure of a Riemannian manifold that is topologically the open half plane. From results in \cite{pangrushin, panwei} we know that $\mathbb G$ with this weight  is $RCD$ (The $RCD$ condition is stronger than  the original curvature-dimension condition and rules out non-Riemannian Finsler structures and Banach spaces). The doubling of this space along its boundary is a sub-Riemannian manifold and there is  no  weight such that the doubling is a metric measure space that satisfies  $RCD$.  
Theorem \ref{th:re} doesn't apply for this example since $\mathbb G$ doesn't have a {\it smooth} boundary. \smallskip

\noindent
{\it \color{black} Question: Is there  a notion of boundary for collapsed  $RCD(K,N)$ metric measure spaces such that, under correct assumption, a doubling theorem, or more generally a gluing theorem, holds. }\smallskip
\noindent\end{remark}
 For non-collapsed $RCD(K,N)$ spaces, i.e. $\m=\mathcal H^N$,  notions of boundary were introduced in \cite{Kap-Mon19, GP-noncol}. In this context a doubling or gluing theorem is expected to be true in the same form as for Alexandrov spaces. For $RCD$ spaces where the corresponding metric space also satisfies locally an upper curvature bound in the sense of Alexandrov  a well-defined notion of boundary has been introduced in \cite{kkk}. Hence for such spaces it seems feasable to give a partial  answer to the previous question. 
\smallskip

Very recently Reiser and Wraith began to investigate glued spaces in connection with $k$-intermediate Ricci curvature bounds $\ric_k\geq K$ \cite{rw, rwnew}.  The  $k$-intermediate Ricci curvatures  interpolate between sectional curvature ($k=1$) and Ricci curvature ($k=n-1$). Reiser and Wraith prove a generalization of one direction in the theorems of Perelman and Kosovski\u{\i} under intermediate lower Ricci curvature bounds where the assumption on the boundary is the same as in \cite{kosovski, perelmanlarge}, i.e. $\Pi_0+\Pi_1\geq 0$ on $\partial M_0\simeq \partial M_1$. They show the glued space of Riemannian manifolds with $k$-intermediate Ricci curvature  is the uniform limit of smooth Riemannian manifolds with $k$-intermediate Ricci curvature lower bounds up to an arbitrarily small $\epsilon$ error. In \cite{ketterermondino} the author and Andrea Mondino propose a synthetic definition of intermediate Ricci lower curvature bounds. Stability of this definition under Gromov-Hausdorff convergence is  an open problem and so we don't know  whether the glued spaces that Reiser and Wraith consider belong to this class. But it is clear that $\ric_k\geq K$ implies $\ric\geq K$ for all $k\in \{1, \dots, n-1\}$ and in particular the glued spaces of Reiser and Wraith with $k$-interemediate lower Ricci curvature bounds satisfy a curvature-dimension condition $CD(K,n)$.  Hence Theorem \ref{th:re} implies the following result.  
\begin{corollary}
The  boundary condition $\Pi_0+\Pi_1\geq 0$ in the theorem of Reiser and Wraith on gluing constructions under $k$-intermediate lower Ricci curvature bounds is  necessary. 
\end{corollary}

Theorem \ref{main3}  was also obtained independently   in \cite{reitri}. They prove a more general  version of Theorem \ref{main3} and use this  to  construct new examples of weighted Riemannian manifolds with postive Bakry-Emery curvature via surgergy.
The proof in \cite{reitri}  follows the work of Perelman. Our approach is based on \cite{sch, kosovski}.
\\

\noindent
The rest of the article is structured as follows. 

In Section 2 we recall the definition of the Bakry-Emery tensor for a weighted Riemannian manifold and its properties, the curvature-dimension condition in the sense of Lott-Sturm-Villani for a metric measure space, some facts about semi-concave functions, the definition of Alexandrov spaces and the $1D$-localisation technique. 

In Section \ref{sec:gluing} we recall the construction of a glued space between Riemannian manifolds $M_0$ and $M_1$. We also recall the proof of the preservation of lower Ricci curvature bounds for glued spaces given by \cite{kosovski, sch}. Finally we define the glued space between weighted Riemannian manifolds. 

In Section \ref{sec:4}  we first recall the definition of warped products and the formula for the Ricci curvature of warped product. We then prove Theorem \ref{main2} and Theorem \ref{main3}. 

In Section \ref{sec: Second application} we show how the $1D$ localisation technique is applied to obtain an important property of geodesics in the glued space. We use that to show that under our boundary assumption the tautological extension of semi-concave functions on $M_0$ and $M_1$ to the glued space is still semi-concave.  We prove Corollary \ref{cor1}.

Finally in  Section \ref{sec:6} we prove Theorem \ref{th:re}.

\subsubsection*{Acknowledgments}
Parts of this work were done when the author stayed   as a Longterm Visitor  at the Fields Institute in Toronto during the Thematic Program on {\it Nonsmooth Riemannian and Lorentzian Geometry}. I want to thank the organizers of the program and  the Fields Institute  for providing  an excellent research environment.  I  want to thank Philipp Reiser for stimulating discussions about gluing constructions and surgery and for remarks on an earlier version of this article.  I also want to thank Alexander Lytchak for suggesting Theorem \ref{th:re}. Finally, I am very grateful to the unknow referee whose insightful comments and remarks lead to major improvements in the final version of this article.
\section{Preliminaries}
\subsection{Bakry-Emery curvature condition}
Let $(M,g)$ be a connected Riemannian manifold with  boundary $\partial M$ and $d_g$  the induced intrinsic distance. Assume $M$ is equipped with  the measure  $\Phi  \vol_M$ where $\Phi\in C^{\infty}(M)$ and $\Phi\geq 0$.  We set $\mathring M= M\backslash \Phi^{-1}(\{0\})$. We call  the triple $(M,g, \Phi)$ a weighted Riemannian manifold. 
\begin{definition}[\cite{baem}] \label{def:baem} Let $N\geq 1$.
If  $N>n=\dim_M$,  for $p\in \mathring M$ and $v\in T_pM$ the Bakry-Emery $N$-Ricci tensor is
\begin{align*}
\ric_g^{\Phi, N}|_p(v,v)
= \ric_g|_p(v,v) - (N-n) \frac{ \nabla^2 \Phi^{\frac{1}{N-n}}|_p(v,v)}{\Phi^{\frac{1}{N-n}}(p)}.
\end{align*}
If $N=n$, then 
\begin{align*}
\ric_g^{\Phi, n}|_p(v,v) = \begin{cases} \ric_M|_p(v,v) - \nabla^2 \log \Phi |_p(v,v)
&\mbox{
if $g( \nabla \log \Phi _p, v)=0$}, \\
-\infty&\mbox{ otherwise.}
\end{cases}\end{align*} 
Finally, we set $\ric_g^{\Phi, N}\equiv-\infty$ if $1\leq N<n$. 
\end{definition}
\begin{fact}\label{firstfact} Let $\eta \in \R$ and $N> \dim_M=n$. 
If $M$ has sectional curvature bounded from above $\overline \kappa>0$ and $(M,g, \Phi)$ satisfies $\ric_g^{\Phi,N}\geq (N-1) \eta$, then 
$$\nabla^2 \Phi^{\frac{1}{N-n}}+ \theta\Phi^{\frac{1}{N-n}}\leq 0 \mbox{ on } \mathring M \mbox{ 
where $\theta :=\min\{-\overline \kappa, \eta\} < 0.$ }$$
\end{fact}

\begin{proof}
To see this observe
$(N-n) \frac{\nabla^2\Phi^{\frac{1}{N-n}}}{\Phi^{\frac{1}{N-n}}}+(N-1)\eta \leq \ric_g\leq  (n-1)\overline\kappa .$
\end{proof}
\begin{corollary}\label{cor:pos}
If  $\ric_g^{\Phi, N}\geq (N-1) \eta$ and there exists $q\in M$ such that $\Phi(q)>0$, then $\Phi>0$ on $M\backslash \partial M$. 
\end{corollary}
\begin{proof}
Assume there exists $p\in M\backslash \partial M$ such that $\Phi(p)=0$.   Then $p$ is also a critical point, i.e. $\nabla \Phi|_p=0$. Since $\nabla^2 \Phi^{\frac{1}{N-n}}+ \theta \Phi^{\frac{1}{N-n}}\leq 0$, it follows  that $\Phi\equiv 0$ in a neighborhood of $p$. 
Hence $\Phi^{-1}(\{0\})$ is  open  in $M\backslash \partial M$ that is   connected. It follows that $\Phi\equiv 0$. 
\end{proof}
\begin{example}
Consider $(I, \Phi)$ where $I\subset \R$  is an interval and $\Phi: I\rightarrow [0,\infty)$ is smooth.  In this case the   Bakry-Emery $N$-Ricci curvature is bounded from below by $K$ if and only if $\frac{d^2}{dt^2} \Phi^{\frac{1}{N-1}}+ \frac{K}{N-1} \Phi^{\frac{1}{N-1}}\leq 0$ on $\mathring I$.  Moreover, it follows $\Phi^{-1}(\{0\})\subset \partial I$. 
\end{example}
\subsection{Curvature-dimension condition}
For $\kappa\in \mathbb{R}$ let $\sin_{\kappa}:[0,\infty)\rightarrow \mathbb{R}$ be the solution of 
$
v''+\kappa v=0, \ v(0)=0 \ \ \& \ \ v'(0)=1.
$
 $\pi_\kappa\in \left\{\frac{\pi}{\sqrt{\kappa}}, \infty\right\}$ is the diameter of a simply connected space  $\mathbb S^2_\kappa$ of constant curvature $\kappa$.
%
\smallskip

For $K\in \mathbb{R}$, $N\in (0,\infty)$ and $\theta\geq 0$ we define the \textit{distortion coefficient} as
\begin{align*}
t\in [0,1]\mapsto \sigma_{K,N}^{(t)}(\theta)=\begin{cases}
                                             \frac{\sin_{K/N}(t\theta)}{\sin_{K/N}(\theta)}\ &\mbox{ if } \theta\in [0,\pi_{K/N}),\\
                                             \infty\ & \ \mbox{otherwise}.
                                             \end{cases}
\end{align*}
One sets $\sigma_{K,N}^{(t)}(0)=t$.
Moreover, for $K\in \mathbb{R}$, $N\in [1,\infty)$ and $\theta\geq 0$ the \textit{modified distortion coefficient} is defined as
\begin{align*}
t\in [0,1]\mapsto \tau_{K,N}^{(t)}(\theta)=\begin{cases}
                                            \theta\cdot\infty \ & \mbox{ if }K>0\mbox{ and }N=1,\\
                                            t^{\frac{1}{N}}\left[\sigma_{K,N-1}^{(t)}(\theta)\right]^{1-\frac{1}{N}}\ & \mbox{ otherwise}
                                           \end{cases}\end{align*}
                                           where $0\cdot \infty=0$. 

Let $(X,d)$ be a complete separable metric space equipped with a locally finite Borel measure $\m$. We call the triple $(X,d,\m)$ a metric measure (mm) space. 
Let $\L$ be the induced length functional for  continuous curves in  $(X,d)$. A geodesic  is a continuous map $\gamma: [a,b]\rightarrow X$   such that $\L(\gamma)=d(\gamma(a), \gamma(b))$.
The set of  Borel probability measures with finite second moment is $\mathcal P^2(X)$,  the set of probability measures in $\mathcal P^2(X)$ that are $\m$-absolutely continuous is denoted with $\mathcal P^2(X,\m)$ and the subset of measures in $\mathcal P^2(X,\m)$ with bounded support is  $\mathcal{P}_b^2(X,\m)$.
The space $\mathcal P^2(X)$ is equipped with the $L^2$-Wasserstein distance $W_2$.  $(\mathcal P^2(X), W_2)$ is a metric space. 
\smallskip

{The \textit{$N$-Renyi entropy} is defined by
\begin{align*}
S_N(\cdot|\m):\mathcal{P}^2_b(X)\rightarrow (-\infty,0],\ \ S_N(\mu|\m)=\begin{cases}-\int_X \rho^{1-\frac{1}{N}}d\!\m& \ \mbox{ if $\mu=\rho\m$,  }\smallskip\\
0&\ \mbox{ otherwise}.
\end{cases}
\end{align*}}

\begin{definition}[\cite{stugeo2, lottvillani}]\label{def:cd}
A mm space $(X,d,\m)$ satisfies the \textit{curvature-dimension condition} $CD(K,N)$ for $K\in \mathbb{R}$, $N\in [1,\infty)$ if for every pair $\mu_0,\mu_1\in \mathcal{P}_b^2(X,\m)$ 
there exists an $L^2$-Wasserstein geodesic $(\mu_t)_{t\in [0,1]}$ and an optimal coupling $\pi$ between $\mu_0$ and $\mu_1$ such that 
\begin{align}\label{ineq:cd}
S_N(\mu_t|\m)\leq -\int \left[\tau_{K,N}^{(1-t)}(\theta)\rho_0(x)^{-\frac{1}{N}}+\tau_{K,N}^{(t)}(\theta)\rho_1(y)^{-\frac{1}{N}}\right]d\pi(x,y)
\end{align}
where $\mu_i=\rho_id\m$, $i=0,1$, and $\theta= d(x,y)$.
\end{definition}
It will be useful to introduce a localized version of the curvature-dimension condition. 

\begin{definition}\label{locdef}
Given a mm space $(X,d,\m)$  let $U\subset X$ be an open subset. We say $U$ satisfies the condition $CD_{loc}(K,N)$ for $K\in \mathbb{R}$, $N\in [1,\infty)$ if for every pair $\mu_0,\mu_1\in \mathcal{P}_b^2(X,\m)$ with $\mu_i(U)=1$ for $i=0,1$
there exists an $L^2$-Wasserstein geodesic $(\mu_t)_{t\in [0,1]}$  in $X$ and an optimal coupling $\pi$ between $\mu_0$ and $\mu_1$ such that \eqref{ineq:cd} holds. 
\end{definition}
\begin{remark}
{\color{black}
The condition $CD_{loc}(K,N)$, that we propose here  for an open subset $U$  in $X$,  is  non-standard.
The local curvature-dimension condition $CD_{loc}(K,N)$ usually applies to  a metric measure space $(X, d, \m)$   and requires that {\it every} point in $X$ has a neighborhood that satisfies the property in Definition \ref{locdef}. Moreover, if $(X,d,\m)$ is  also essentially non-branching, this condition $CD_{loc}(K,N)$  implies the curvature-dimension condition $CD(K,N)$ \cite{cavmil}.}
\end{remark}
\begin{example}
A metric measure space $(I, \m)$ for an interval $I\subset \R$  and a measure $\m$ satisfies the condition $CD(K,N)$ for $K\in \R$ and $N\in [1, \infty)$  if and only if $\m= \Phi d\mathcal L^1$  and  $\Phi$ is $(K, N)$-concave in the sense that 
$$\frac{d^2}{dt^2} \Phi^{\frac{1}{N-1}} + \frac{K}{N-1} \Phi^{\frac{1}{N-1}}\leq 0\ \mbox{ on } I$$
in the distributional sense, e.g.  $([0,\pi], \sin^{N-1}(r) dr)$ satisfies $CD(N-1,N)$.
\end{example}
\begin{theorem}\label{th:cdbe} Let $(M, g, \Phi)$ be a weighted Riemannian manifold with $\partial M\neq \emptyset$.
$(M, g, \Phi)$ satisfies $\ric^{\Phi, N}_g\geq K$ and $\Pi_{\partial M}\geq 0$ if and only if the metric measure space $(M, d_g, \Phi \vol_M)$ satisfies the condition $CD(K,N)$.
\end{theorem}
\noindent
Here $\Pi_{\partial M}$ is the second fundamental form of $\partial M$.  For the case $\partial M=\emptyset$ see \cite{cms, sturmrenesse}.
In this form the theorem was proved in \cite{han19}. 
\begin{remark}
The set of $CD(K,N)$ spaces with $K\in \R$ and $N\in [1,\infty)$ is closed w.r.t. measured Gromov-Hausdorff convergence \cite{stugeo2, lottvillani}.  Moreover, let $(X_i,d_i,\m_i,o_i)$ converge in pointed GH sense to a pointed metric measure space $(X,d,\m,o)$.  This means that balls of a fixed radius converge in measured GH sense \cite{gmsstability}. It is not hard to see that a condition  $CD_{loc}(K,N)$ for the balls $B_R(o_i)$ implies the condition $CD_{loc}(K,N)$ for the ball $B_R(o)$ in the limit, for instance compare this with the proof of Theorem 7.1 in \cite{kettererlp}.
\end{remark}

\subsection{Semiconcave functions}
We recall some  facts about concave functions \cite{plaut, simonconvexity}.

Let
$u:[a,b]\rightarrow \R$ be concave.
Then $u$ is lower semi continuous and continuous on $(a,b)$.
 The right and left  derivative
$
\frac{d}{dr^+} u (r)$ and 
$\frac{d}{dr^-} u (r)$
exist in $\R\cup\{\infty\}$ and $\R\cup \{-\infty\}$ respectively
for  $r\in [a,b]$ with values in $\R$ if $r\in (a,b)$. Moreover $\frac{d}{dr^+}u(r)\leq \frac{d}{dr^-}u(r)$ $\forall r\in (a,b)$ and $\frac{d}{dr^{+/-}}u(r)\downarrow$ in $r$. 

\begin{definition}
Let $f: [a,b]\rightarrow \R$ be continuous on $(a,b)$, and let $F:[a,b]\rightarrow \R$ be such that
$F''=f$ on $(a,b)$.
For a function $u: [a,b]\rightarrow \R$ 
we write $u''\leq f$ on $(a,b)$ if $u-F$ is concave on $(a,b)$. We say $u$ is $f$-concave.
\end{definition}
\noindent
We say $u$ is semiconcave if for any $r\in (0,\theta)$ we can find $\epsilon>0$ and $\lambda\in \R$ such that $u$ is $\lambda$-concave on $(r-\epsilon,r+\epsilon)$.

Consider $u: [a,b]\rightarrow (0,\infty)$ that satisfies
\begin{align}\label{kuconcavity1}
u\circ\gamma(t)\geq \sigma_{\theta}^{(1-t)}(|\dot{\gamma}|)u\circ \gamma(0) + \sigma_{\theta}^{(t)}(|\dot{\gamma}|)u\circ\gamma(1)
\end{align}
for any constant speed geodesic $\gamma:[0,1]\rightarrow [a,b]$. 
It follows that $u$ is lower semi continuous and continuous on $(a,b)$.
 Let
$U(t):=\int_a^bg(s,t) u(s) ds$. Then $U$ satisfies $U''=-u$ on $(a,b)$ where $g(s,t)$ is the Green function of the interval $(a,b)$.
If $u$ satisfies \eqref{kuconcavity1} for every constant speed geodesic $\gamma:[0,1]\rightarrow [a,b]$, then
\begin{align}\label{die}
u''+\theta u\leq 0\mbox{ on }(a,b)
\end{align}
in the sense that $u-\theta U$ is concave on $(a,b)$.

The next lemma is Lemma 2.17 in \cite{kakest}.
\begin{lemma}
Let $u:[a,b]\rightarrow \R$ be lower semi-continuous and continuous on $(a,b)$ such that $u''+\theta u\leq 0$ on $(a,b)$ in the sense of the definition above.
Then $u$ satisfies \eqref{kuconcavity1} for every constant speed geodesic $\gamma:[0,1]\rightarrow [a,b]$.
\end{lemma}
\begin{remark}
Let $u:[a,b]\rightarrow [0,\infty)$ be a function. It is easy to check that the metric measure space $([a,b], ud\mathcal L^1)$ satisfies $CD(K,N)$ if and only if $u^{\frac{1}{N-1}}$ satisfies \eqref{kuconcavity1} with $\kappa= \frac{K}{N-1}$. 
\end{remark}
\begin{lemma}\label{lem:local}
If $u$ satisfies \eqref{kuconcavity1} for every constant speed geodesic $\gamma:[0,1]\rightarrow [a,b]$ of length less than $\theta<b-a$, then $u$ satisfies \eqref{kuconcavity1}.
\end{lemma}
\begin{lemma}\label{importantlemma}
Let  $u: (a,b)\rightarrow \R$ be continuous and $c\in (a,b)$ such that $u''\leq - ku$ on $(a,c)\cup (c,b)$.
Then $u''\leq -k u$ on $(a,b)$ if 
$
\frac{d}{dr^-} u (c)\geq \frac{d}{dr^+} u (c).
$
\end{lemma}
\subsection{Alexandrov spaces}\label{intro:Alex}
Let $\md_\kappa:[0,\infty)\rightarrow [0,\infty)$ be the solution of 
\begin{align*}
v''+ \kappa v=1 \ \ \ v(0)=0 \ \ \&\ \ v'(0)=0.
\end{align*}
\begin{definition}
A complete geodesic metric space $(X,d)$  has curvature bounded  below by $\kappa\in\mathbb{R}$ in the sense of Alexandrov  if for any  constant speed geodesic $\gamma : [0,\L(\gamma)]\to X$ {and any  point $y\in X$} 
such that 
$
d(y,\gamma(0))+\L(\gamma)+d(\gamma(l),y)<2\pi_\kappa
$
it holds that 
\begin{align*}
\left[\md_{\kappa}(d_y\circ\gamma)\right]''+\md_{\kappa}(d_y\circ\gamma)\leq 1.
\end{align*}
If $(X,d)$ has curvature bounded from below for some $\kappa\in \R$ in the sense of Alexandrov, we say that $(X,d)$ is an Alexandrov space.
\end{definition}

\begin{remark}
Alexandrov spaces are nonbranching: if $\gamma, \tilde \gamma: (0,1) \rightarrow X$ are geodesics with $\gamma|_{(0,\epsilon)}= \tilde \gamma|_{(0, \epsilon)}$ for some $\epsilon>0$, then $\gamma=\tilde \gamma$. 
\end{remark}
\begin{remark}
 The set of $n$-dimensional  Alexandrov spaces with curvature bounded from below by $\kappa$ is closed w.r.t. Gromov-Hausdorff convergence.
\end{remark}
\begin{theorem}[\cite{palvs}]\label{th:petrunincd}
Let $X$ be an $n$-dimensional Alexandrov space with curvature bounded  below by $\kappa$. 
Then $(X,d_X,\mathcal H^n_X)$ satisfies $CD(\kappa(n-1),n)$.
\end{theorem}
\begin{remark}
Let $(X,d)$ be an $n$-dimensional  Alexandrov space. There exists a  unique Gromov-Hausdorff  blow up tangent cone $C_pX$ at $p\in X$. The blow up tangent cone at $p$ coincides with the metric cone $C(\Sigma_pX)$ over   the space of directions $\Sigma_pX$ at $p$ equipped with the angle metric \cite{bbi}. 
Let $\gamma:[0,l]\rightarrow X$ be a unit speed geodesic such that  $\gamma(t_0)=p$ for $t_0\in [0,l]$. Then, there exists a unique $v\in \Sigma_p X$ such that $\gamma(t_i)\rightarrow v$ for any sequence $t_i\rightarrow t_0$.
The limit $\gamma(t_i)\rightarrow v$  is  understood in  the following sense: For every $i\in \N$ there exists an $\epsilon_i$-Gromov-Hausdorff-approximation $\psi_i: (t_i^{-1}B_{2t_i}(p),p)\rightarrow B_{2}(o)\subset (C_pX,o)$. Then it is required that $\psi_i(\gamma(t_i))\rightarrow v$. We say $\gamma(t_i)$ converges to $v$ in GH sense. For an overview on Alexandrov spaces with lower curvature bounds we refer to \cite{petsem}.
\end{remark}

\subsubsection{ Semiconcave functions on Alexandrov spaces}
Let $X$ be an Alexandrov  space with curvature bounded from below by $\underline \kappa$ and let $\Omega\subset X$ be open. A function $f:\Omega\rightarrow \R$ is $\lambda$-concave if $f$ is { locally Lipschitz} and $f\circ\gamma:[0,\mbox{L}(\gamma)]\rightarrow \R$ is $\lambda$-concave for every constant speed geodesic $\gamma:[0, \mbox{L}(\gamma)]\rightarrow \Omega$. $f$ is semiconcave if for every $p\in \Omega$ there exists an open neighborhood $U$ of $p$ such that $f|_U$ is { $\lambda$-concave for some  $\lambda\in \R$.}  We adopt the following terminology from \cite{albi}. A function $f: \Omega \rightarrow [0,\infty)$ that satisfies \eqref{kuconcavity1} along any constant speed geodesic $\gamma:[0,1]\rightarrow \Omega$ is called {\it $\mathcal F \theta$-concave} in $\Omega$.
\begin{fact}
Let $(M, g_M)$ be a  Riemannian manifold  with boundary and let $\Omega \subset M\backslash \partial M$ be open.  A smooth function $f$ is $\mathcal F\theta$-concave in $\Omega$ if and only if $\nabla^2 f+ \theta f g_{\sM}\leq 0$ on $\Omega$.
\end{fact}
Let $X$ be an Alexandrov space and {$f: X \rightarrow \R$  $\lambda$-concave.} The limit
\begin{align*}
\lim_{r_i\downarrow 0} \frac{f\circ \gamma(r_i) - f\circ \gamma(0)}{r_i}=\frac{d}{dr^+}( f\circ \gamma)(0)=: df_p(\dot{\gamma})=:df(\dot{\gamma}) 
\in \R
\end{align*}
exists for every constant speed geodesic $\gamma:[0,\L(\gamma)]\rightarrow X$ with $\gamma(0)=p$,  for every $r_i\downarrow0$, and for every $p\in X$. The map
$df_p: C_pX\rightarrow \R$  is called the differential of $f$ in $p\in X$. If $(X,d)= (M,d_M)$ is a Riemannian manifold where $d_M$ is the induced distance and $f$ is smooth,  this coincides with the classical notion of the differential.

\subsection{$1D$ localisation}\label{subsec:1Dlocalisation}
Let $(X,d,\m)$ be a locally compact metric measure space that is (essentially) nonbranching. We  assume that $\supp\m =X$.
Let $u:X\rightarrow \mathbb{R}$ be a $1$-Lipschitz function. Then 
\begin{align*}
\Gamma_u:=\{(x,y)\in X\times X : u(x)-u(y)=d(x,y)\}
\end{align*}
is a  $d$-cyclically monotone set, and one defines $$\Gamma_u^{-1}=\{(x,y)\in X\times X: (y,x)\in \Gamma_u\}.$$
If $\gamma:[a,b]\rightarrow X$ is a geodesic such that $(\gamma(a),\gamma(b))\in \Gamma_u$ then $(\gamma(t),\gamma(s))\in \Gamma_u$ for $a<t\leq s<b$.  It is therefore natural to consider the set $G$ of unit speed transport geodesics $\gamma:[a,b]\rightarrow \R$ such that $(\gamma(t),\gamma(s))\in \Gamma_u$ for $a\leq t\leq s\leq b$. 
The union $\Gamma_u\cup \Gamma_u^{-1}$ defines a relation $R_u$ on $X\times X$, and $R_u$ induces a {\it transport set with endpoints} 
$$\mathcal T_u:= P_1(R_u\backslash \{(x,y):x=y\})\subset X$$ 
where $P_1(x,y)=x$. For $x\in \T_u$ one defines $\Gamma_u(x):=\{y\in X:(x,y)\in \Gamma_u\}, $
and similarly $\Gamma_u^{-1}(x)$ as well as $R_u(x)=\Gamma_u(x)\cup \Gamma_u^{-1}(x)$. Since $u$ is $1$-Lipschitz, 
 $\Gamma_u, \Gamma_u^{-1}, R_u\subset X^2$ and  $\Gamma_u(x), \Gamma_u^{-1}(x), R_u(x)\subset X$ are closed.

The {\it transport set without branching} $\mathcal T^b_u$ associated to $u$ is  defined as 
\begin{align*}
\mathcal T^b_u=\left\{ x\in \mathcal T_u: \forall  y,z\in R_u(x) \Rightarrow  (y,z)\in R_u\right\}.
\end{align*}
The sets $\T_u$ and $\T_u\backslash \T^b_u$ are $\sigma$-compact, and $\T^b_u$ and $R_u\cap \T^b_u\times \T^b_u$ are Borel sets.
In \cite{cavom} Cavalletti shows that $R_u$ restricted to $\T^b_u\times \T^b_u$ is an equivalence relation. 
Hence, from $R_u$ one obtains a partition of $\mathcal T^b_u$ into a disjoint family of equivalence classes $\{X_{\gamma}\}_{\gamma\in Q}$. Moreover, $\T^b_u$ is also $\sigma$-compact. 

Every $X_{\gamma}$ is isometric to some interval  $I_\gamma\subset\mathbb{R}$ via a distance preserving map $\gamma:I_\gamma \rightarrow X_{\gamma}$.  The map $\gamma:I_\gamma\rightarrow X$ extends to a geodesic  on $\overline I_{\gamma}$ that is arclength parametrized and that we also denote $\gamma$. Let $(a_\gamma,b_\gamma)$ be the interior of $\overline I_\gamma$ where $a_\alpha, \beta \in \R\cup\{\pm\}$. 

The set of equivalence classes $Q$ has a measurable structure such that $\mathfrak Q: \T^b_u\rightarrow Q$ is a measurable map. 
We set $\mathfrak q:= \mathfrak Q_{\#}\m|_{\mathcal T^b_u}$. 

{A measurable section of the equivalence relation $R$ on $\T^b_u$ is a measurable map $s: \T^b_u\rightarrow \T^b_u$ such that $R_u(s(x))=R_u(x)$ and $(x,y)\in R_u$ implies $s(x)=s(y)$. In \cite[Proposition 5.2]{cavom} Cavalletti shows there exists a measurable section $s$ of $R$ on $\T^b_u$. Therefore, one can identify the measurable space $Q$ with the measurable set $\{x\in \T^b_u: x=s(x)\}\subset X$ equipped with the induced measurable structure. Hence $\mathfrak q$ is a Borel measure on $X$. By inner regularity there exists a $\sigma$-compact set $Q'\subset X$ such that $\mathfrak q(Q\backslash Q')=0$ and in the following we will replace $Q$ with $Q'$ without further notice. {$\gamma\in Q$ is parametrized such that $\gamma(0)=s(x)$.} The functions $\gamma\in Q \mapsto a_{\gamma}, b_{\gamma}\in \R$ are measurable. We set $Q^l=\{l\geq |a_\gamma|, |b_\gamma|\geq \frac{1}{l}\}$.}
%
\begin{lemma}[Theorem 3.4 in \cite{cavmon}]\label{somelemma}
Let $(X,d,\m)$ be an essentially nonbranching $CD^*(K,N)$ space for $K\in \R$ and $N\in (1,\infty)$ with $\supp \m=X$ and $\m(X)<\infty$.
Then, for any $1$-Lipschitz function $u:X\rightarrow \R$, it holds $\m(\T_u\backslash \T^b_u)=0$.
\end{lemma}
For $\mathfrak q$-a.e. $\gamma\in Q$ it was proved in \cite{cavmil} (Theorem 7.10) that 
\begin{align*}
R_u(x)=\overline{X_\gamma}\supset X_\gamma \supset (R_u(x))^{\circ} \ \ \forall x\in \mathfrak Q^{-1}(\gamma).
\end{align*}
where $(R_u(x))^\circ$ denotes the relative interior of the closed set $R_u(x)$. 
%
\begin{theorem}[{\cite[Theorem 3.3]{cav-mon-lapl-18}}]
Let $(X,d,\m)$ be
a complete, proper, geodesic metric measure space 
with $\supp\m =X$ and $\m$ $\sigma$-finite. Let $u:X\rightarrow \mathbb{R}$ be a $1$-Lipschitz function,  let $(X_{\gamma})_{\gamma\in Q}$ be the induced partition of $\mathcal T^b_u$ via $R_u$, and let $\mathfrak Q: \T^b_u\rightarrow Q$ be the induced quotient map as above.
Then, there exists a unique strongly consistent disintegration $\{\m_{\gamma}\}_{\gamma\in Q}$ of $\m|_{\T^b_u}$ w.r.t. $\mathfrak Q$. 
\end{theorem}
Define the ray map 
\begin{align*}
\mathfrak G:  \mathcal V\subset Q\times \R\rightarrow X\ \mbox{ via }\
\mbox{graph}(g)=\{ (\gamma, t,x) \in Q\times \R\times X:  \gamma(t)=x\}.  
\end{align*}
\begin{itemize}
\item 
The map $\mathfrak G$ is Borel measurable,  and 
by definition $\mathcal V= \mathfrak G^{-1}(\T^b_u)$. 
\item $\mathfrak G(\gamma,\cdot)=\gamma: (a_\gamma, b_\gamma)\rightarrow X$ is a geodesic, 
\item $\mathfrak G:\mathcal V\rightarrow  \T^b_u$ is bijective and its inverse is given by $\mathfrak G^{-1}(x)=( \mathfrak Q(x), \pm d(x,\mathfrak Q(x)))$.
\end{itemize}
\begin{theorem}[{\cite[Theorem 3.5]{cav-mon-lapl-18}}]\label{th:1dlocalisation}
Let $(X,d,\m)$ be an essentially nonbranching $CD^*(K,N)$  space with $\supp\m=X$, $\m$ $\sigma$-finite, $K\in \R$ and $N\in (1,\infty)$.

Then, for any $1$-Lipschitz function $u:X\rightarrow \R$ there exists a disintegration $\{\m_{\gamma}\}_{\gamma\in Q}$ of $\m$ that is strongly consistent with $R^b_u$. 

Moreover, for $\mathfrak q$-a.e. $\gamma\in Q$, $\m_{\gamma}$ is a Radon measure with $\m_{\gamma}=h_{\alpha}\mathcal{H}^1|_{X_{\alpha}}$ and $(X_{\gamma}, d_{X_\gamma}, \m_{\gamma})$ verifies the condition $CD(K,N)$.

More precisely, for $\mathfrak q$-a.e. $\gamma\in Q$ it holds that
\begin{align}\label{kuconcave}
h_{\gamma}(c_t)^{\frac{1}{N-1}}\geq \sigma_{K/N-1}^{(1-t)}(|\dot c|)h_{\gamma}(\gamma_0)^{\frac{1}{N-1}}+\sigma_{K/N-1}^{(t)}(|\dot c|)h_{\gamma}(\gamma_1)^{\frac{1}{N-1}}
\end{align}
for every geodesic $c:[0,1]\rightarrow (a_\gamma,b_\gamma)$.
\end{theorem}
\begin{remark} The property
\eqref{kuconcave} yields that $h_{\gamma}$ is locally Lipschitz continuous on $(a_\gamma,b_\gamma)$ \cite[Section 4]{cavmon}, and that $h_{\gamma}:\R \rightarrow (0,\infty)$ satisfies
\begin{align*}
\frac{d^2}{dr^2}h_{\gamma}^{\frac{1}{N-1}}+ \frac{K}{N-1}h_{\gamma}^{\frac{1}{N-1}}\leq 0 \mbox{ on $(a_\gamma, b_\gamma)$} \mbox{ in distributional sense.}
\end{align*}
\end{remark}
%
\subsubsection{Characterization of curvature bounds via $1D$ localisation}
\begin{definition}
Let $(X,d_X,\m_X)$ be an essentially nonbranching mm space with $\m(X)=1$, let $K\in \R$ and $N\geq 1$, and let $u:X\rightarrow \R$ be  $1$-Lipschitz. We say that $(X,d_X,\m_X)$ satisfies the condition $CD_u^1(K,N)$ if there exist subsets $X_\gamma\subset X$, $\gamma\in Q$, such that
\begin{itemize}
\item[(i)] There exists a disintegration of $\m_{\mathcal T_u}$ on $(X_\gamma)_{\gamma\in Q}$: 
\begin{align*}
\m|_{\mathcal T_u}=\int_{X_\gamma} \m_\gamma d\mathfrak q(\gamma)\ \mbox{ with }\m_\gamma(X_\gamma)=1\ \mbox{for $\mathfrak q$-a.e. }\gamma\in Q.
\end{align*}
\item[(ii)] For $\mathfrak q$-a.e. $\gamma\in Q$ the set $X_\gamma$ is the image $\mbox{Im}(\gamma)$ of a geodesic $\gamma:I_\gamma\rightarrow X$ for an interval $I_\gamma\subset \R$.
\item[(iii)] The metric measure space $(X_\gamma, d_{X_\gamma},\m_\gamma)$ satisfies the condition $CD(K,N)$.
\end{itemize}

A metric measure space $(X,d_X,\m_X)$ satisfies the condition $CD^1_{Lip}(K,N)$ if it satisfies the condition $CD^1_u(K,N)$ for every $1$-Lipschitz function $u$.
\end{definition}
\begin{remark}
From the previous subsection it is immediately clear that the condition $CD(K,N)$ implies the condition $CD^1_{Lip}(K,N)$. 
\end{remark}
\begin{theorem}[\cite{cavmil}]\label{thm:cavmil}
If an essentially nonbranching metric measure space $(X,d_X,\m_X)$ with finite measure satisfies the condition $CD_{Lip}^1(K,N)$ for $K\in \R$ and $N\in [1,\infty)$ then it satisfies the condition $CD(K,N)$.
\end{theorem}

\section{Glued spaces}\label{sec:gluing}
We   consider Riemannian manifolds $(M_i,g_i)$, $i=0,1$, with boundary. 
\begin{assumption}\label{ass:1_0} 
Let $Y_i$ be a  closed subset in $\partial M_i$. Assume
there exists a Riemannian isometry $\mathcal I:   Y_0\rightarrow  Y_1$.  Moreover, we assume
\begin{enumerate}
\item
  $\exists \overline \kappa, \underline \kappa \in \R$ such that $\underline \kappa\leq \mbox{sec}_{M_i}\leq \overline \kappa$.
\smallskip
\item
$Y_i$ is a compact, connected component of $\partial M_i$, $i=0,1$.
\end{enumerate}
\end{assumption}
\begin{assumption}\label{ass:1}
 We assume that
\begin{align*}\Pi:=(\Pi_0+\Pi_1)|_{Y_0\simeq Y_1}\geq 0 
, \ i=0,1,\end{align*}
where $\Pi_i= \Pi_{\partial M_i}$ is the 2nd fundamental form w.r.t. the inward pointing normal vector. 
\end{assumption}
\smallskip

The topological glued space of $ M_0$ and $M_1$ along $Y_0$ and $Y_1$ w.r.t. $\mathcal I$ is defined as the quotient space $( M_0\dot \cup  M_1)/R=:M$ of the disjoint union $M_0\dot \cup M_1$ where 
\begin{align*}
x\sim_ R y \ \ \mbox{ if and only if }\ \ \begin{cases} \mathcal I(x)=y& \mbox{ with }x\in Y_0, y\in Y_1, \\
x=y& \mbox{  otherwise.}
\end{cases}
\end{align*}
The  space ${\color{white} \hat  .} M$ is a topological manifold with boundary. Moreover, ${\color{white} \hat .} M\backslash \partial {\color{white} \hat .} M$  can be equipped with a differentiable structure   such that $ M_0\backslash \partial M_0$ and $ M_1\backslash \partial M_1$ are smooth submanifolds \cite{hirschdt}.

The equivalence relation $R$ induces a pseudo distance on ${\color{white} \hat .} M$ as follows.  First, one defines an extended metric $d$ on $M_0\dot \cup M_1$ via $d(x,y)=d_{g_i}(x,y)$ if $x,y\in M_i$ for some $i\in \{0,1\}$ and $d(x,y)=\infty$ otherwise. $d_{g_i}$ is the intrinsic distance w.r.t. $g_i$.
Then,  $\forall x,y\in M_0\dot \cup M_1$ we define
\begin{align*}
\hat d(x,y) = \inf  \sum_{i=0}^{k-1} d(p_{i},q_i)
\end{align*}
where the infimum runs over all collection of tuples 
 $\{(p_i,q_{i})\}_{i=0,\dots ,k-1}\subset {\color{white} \hat .} M\times {\color{white} \hat .} M$ for some $k\in N$
 such that $q_i\sim_R p_{i+1}$, for all $i=0,\dots, k-1$ and $x=p_0, y=q_k$. One can show that $x\sim_R y$ if and only if $\hat d(x,y)=0$ (note that for a gluing construction over metric spaces the latter is not true in general).

The {\it metric} glued space between $M_0$ and $M_1$ w.r.t. $\mathcal I:Y_0 \rightarrow Y_1$ is the metric space defined as 
$$M_0\cup_{\mathcal I} M_1:=(M, {\color{white} \hat .}{d}).$$

One can construct the distance $d$ also as follows. Since $\mathcal I$ is an isometry, there exists a continuous Riemannian metric $g$ on ${\color{white} \hat .} M$ such that $g|_{M_i}= g_i$. A length structure on ${\color{white} \hat .} M$ is defined via 
$$\L_{g}: \gamma\mapsto \int_a^b |\gamma'(t)|_g dt$$
for curves $\gamma:[a,b]\rightarrow {\color{white} \hat .} M$ such that there exist points $a=t_0\leq \dots \leq t_n=b$ with $\gamma|_{[t_{k-1},t_k]}\in C^1([t_{k-1}, t_k], M_i)$ for $i=0$ or $i=1$ and for all $k\in \{1, \dots, n\}$. 
Then  the  induced intrinsic distance coincides with $d$. 
\begin{remark}
The metric $g$ yields a  Riemannian volume  $\vol_g$ on ${\color{white} \hat .} M$.
\end{remark}
\begin{theorem}[\cite{kosovskiigluing, kosovski}] \label{th:basicglue} Assume $\mbox{sec}_{g_i}\geq \underline\kappa\in \R$. Then there exists a family of  Riemannian metrics $(g^\delta)_{\delta>0}$  on $M$ such that 
\begin{enumerate}
\item $ g^\delta$ coincides with $ g$ outside of $B_\delta( Y_0 \simeq  Y_1)$, 
\smallskip
\item $ g^\delta$ converges uniformly to $ g$ as $\delta \downarrow 0$. 
\smallskip
 \item  $\mbox{Sec}_{g^\delta}\geq \underline\kappa-\epsilon(\delta)$ with $\epsilon(\delta)\rightarrow 0$ as $\delta \downarrow 0$.
\end{enumerate}
 In particular,  if $\Pi_i\geq 0$ on $\partial M_i\backslash Y_i$, then the  metric glued space $M_0\cup_{\mathcal I} M_1$ is an Alexandrov space with curvature bounded from below by $\underline \kappa$.
\end{theorem}

Let $p\in Y_0 \simeq Y_1$. The tangent spaces $T_pM_0$ and $T_pM_1$ at $p$ w.r.t. $M_0$ and $M_1$ are isometric to the Euclidean halfspace $\mathbb H= \R_{\geq 0}\times \R^{n-1}$ equipped with the inner product $g_0|_p= g_1|_p$. Moreover, $T_pM_0$ and $T_pM_1$  embed isometrically into the blow up tangent cone $C_pM$ of $M_1\cup_{\mathcal I} M_2$ at $p$. Hence $C_pM$ is isometric to $\R^n$ equipped with $g|_p$. Consequently every point in $M$ is an Alexandrov regular point (see \cite{bbi}).

Since the boundary $\partial M_i$ is smooth, the distance function $f_i=d_{\partial M_i}$ is smooth on $B_\epsilon(p)\cap M_i$ for $\epsilon>0$ sufficiently small. In particular, the differential of $f_i: M_i\rightarrow \R$  exists as a blow up limit in the sense that was introduced for Alexandrov spaces:
$$ df_i|_p: T_pM_i \rightarrow \R, \ \ df_i|_p(v) = g_i|_p(\nabla f_i, v)= \lim_{q_k\rightarrow v} \frac{f(q_k)-f(p)}{d_{M_i}(q_k, p)} .$$
We say $v\in C_pM$ is {\it tangent to the boundary} if $df_0|_p(v)\leq 0$ and $df_1|_p(v) \leq 0$. 
\smallskip

In Theorem \ref{th:basicglue} one can replace lower sectional curvature bounds with lower Ricci curvature bounds. This was  proved by Perelman in \cite{perelmanglue}. A proof based on \cite{kosovskiigluing, kosovski} that also applies to other curvature notions was given by Schlichting \cite{sch}. We also refer to the following publications \cite{burdick, burdickthesis, bww} that contain interesting applications. 
\begin{theorem}\label{th:glue2} Assume $\ric_{g_i}\geq K\in \R$. Then the family of Riemannian metrics $(g^\delta)_{\delta>0}$  satisfies  $\ric_{g^\delta}\geq K-\epsilon(\delta)$ with $\epsilon(\delta)\rightarrow 0$ as $\delta \downarrow 0$.
\end{theorem}
\begin{remark} The uniform convergence of $g^\delta$ to $g$ implies that the corresponding induced intrinsic distances converge  in  pointed GH sense. Moreover $\vol_{g^\delta}$ converges weakly to $\vol_g$. If $\Pi_i\geq 0$ on $\partial M_i \backslash Y_i$, then the metric glued space $M_0\cup_{\mathcal I} M_1$ equipped with the measure $\vol_g$ satisfies the curvature-dimension condition $CD(K,n)$
because of Theorem \ref{th:cdbe} and the  stability of the  condition $CD(K,n)$ under measured GH convergence.
\end{remark}
\paragraph{\bf Proof of Theorem \ref{th:glue2}}\label{subsec:gdelta}
From \cite{sch} we recall the construction of the metric $g^\delta$ that appears in Theorem \ref{th:basicglue} and in Theorem \ref{th:glue2}. 

{\bf (1)}  We cover $M_0$ with  a \red{countable} family of coordinate charts  $$\phi_s=(x_s^1, \dots, x_s^n): U_s\subset M_0\rightarrow  \R^{n-1}\times [0,\infty), \ \ s=1, 2,  \dots,$$  such that $\exists N\in \{1, \dots, S\}$ with $$\mbox{$B_\delta(Y_0)\subset \bigcup_{s=1}^N U_s\subset B_{2\delta}(Y_0)\ $    and \ $Y_0\cap \bigcup_{s=N+1}^{\red{\infty}} U_s=\emptyset$,}$$
and for $s=1, \dots, N$ we have that $(x_s^1, \dots, x_s^{n-1})$ are coordinates of $Y_0$ and $x_s^n= d_{Y_0}$ where $$x\in M_0 \mapsto d_{Y_0}(x)= \inf_{z\in Y_0} d_{M_0}(x,z).$$ 
We choose $\delta>0$ small enough such that $d_{Y_0}$ is smooth on $B_{2\delta}(Y_0)$.
\smallskip

For $p\in Y_0$ let ${\bf I}:= {\bf I}_p$ be the identity operator on $T_pM$ and let ${\bf L}:= {\bf L}_p$ be the self-adjoint operator on $T_pY_0$ induced by $$\langle \cdot, {\bf L}\cdot \rangle=\Pi:= \Pi_0|_{Y_0} +\Pi_1|_{Y_1}.$$
One extends ${\bf L}$ via ${\bf L}N=0$ to $T_pM_0$. 
\smallskip

{\color{black}
In \cite{sch} the author uses the notation $L= \Pi$.  Hence, $L$ should not be confused with ${\bf L}$.  For consistency with \cite{sch} we adapt this notation in this section and write $L=\Pi$.}
\smallskip

\red{
{\bf (2)}
Given a parameter $\delta>0$ we will construct an auxiliary function $\mathcal F_\delta$ following \cite{sch} and \cite{kosovski}. 

We  consider a function $f_\delta: [0, \infty) \rightarrow \R$ with the following properties: 
\begin{enumerate}
\item $f_\delta(x)= 1-\frac{x}{\delta^4}$ on $[0, \delta^4]$, 
\item $-\delta^2 \leq f_\delta(x)\leq 0$ and $f_\delta'(x)\leq \delta$ for $\in [\delta^4, \delta]$, 
\item $f_\delta(x)=0$ for $x\in [\delta, \infty)$, 
\item $\int_0^\delta f_\delta(x) \de x=0$.
\end{enumerate}
Such $f_\delta$ is fairly easy to find. Then, we define for $y, z\in [0, \infty)$. 
$$F_\delta(y)=\int^y_0 f_\delta(x) \de x, \ \ \ \mathcal F_\delta(z)=  \int^z_0 F_\delta(y) \de y.$$
It follows that $F_\delta(y)= \mathcal F_\delta(z)=0$ for $y,z\geq 2\delta$,  $\mathcal F_\delta'= F_\delta, F_\delta'=  f_\delta$ and $\mathcal F_\delta(0)= F'_\delta(0)=0$.
Hence, $F_\delta$ and $f_\delta$ are compactly  supported in $[0, \delta)$.  However, $\mathcal F_\delta$ is not compactly supported in $[0, \delta)$, and it  will become clear shortly that this is  essential.
 \footnote{This is a minor error in the construction of the glued metric presented in \cite{sch} and in \cite{kosovski} as was pointed out to us by the referee. }

We will  perturb $F_\delta$  such that the previous properties are preserved up to controlled errors and the new $\mathcal F_\delta$ is supported in $[0, \delta)$. 
\smallskip

The properties of $f_\delta$ imply $0\leq F_\delta(y)= \int_0^y f_\delta(x) \de x \leq \delta ^4$ for all $y\in [0, \infty)$. Hence $$ \mathcal F_\delta(z)\leq \delta^5 \ \forall z\geq 0, \ \ \ 0\leq \mathcal F_\delta(z)= \int_0^\delta F_\delta(y) \de y= :I(\delta)\leq \delta^5  \mbox{ for } z\geq \delta.$$
 We choose a function $\phi \in C_c^2((0, \infty), [0,1])$ with $\supp \phi \subset [\frac \delta 2 , \delta]$ such that for a constant $C>0$ and for all $y$ we have
\begin{enumerate}
\item $|\phi_\delta(y)|\leq C\delta^3$, 
\item $\int_0^\infty \phi_\delta(y) \de y= I(\delta)$,
\item $-C\delta^2\leq \phi'(y)\leq C\delta^2$, 
\item $\phi''(y) \leq C \delta $.
\end{enumerate}
For instance, we can choose
 $$\phi_\delta(y)= \begin{cases} \frac{4 I(\delta)}{\delta } \sin^4\left( \frac{2\pi}{\delta}(y-\frac \delta 2)\right) & y\in [\frac \delta 2, \delta]\\
 0 & \mbox{otherwise}.
 \end{cases}$$
Then we define 
$$\tilde F_\delta(y)= F_\delta(y)-  \phi_\delta(y), \ \ \ \tilde{\mathcal F}_\delta(z)= \int_0^z \tilde F_\delta(y)\de y.$$
It follows for $z\geq \delta$ that  $$\tilde{\mathcal F}_\delta(z)= \int_0^z \tilde F_\delta(y)\de y= \int_0^z F_\delta(y) \de y - \int_0^z \phi_\delta(y) \de y=0$$
and we have the following properties
\begin{enumerate}
\item $\tilde{\mathcal F}_\delta' =\tilde F_\delta, \tilde F_\delta'$ and $\tilde F_\delta''$ are compactly supported in $[0, \delta)$, 
\item $\tilde{\mathcal F}_{\delta}(0) =\tilde F_\delta(0)=0$, 
\item $\tilde F'_\delta(x)= 1- \frac{x}{\delta^4}$ for $x\in [0, \delta^4]$, 
\item $-C\delta^2\leq \tilde F'_\delta\leq C \delta^2$ and $\tilde F''_\delta\leq C \delta$ on $[\delta^4, \delta]$ for constant a $C$ that does not depend on $\delta>0$.
\end{enumerate}
In the following we write again $f_\delta$, $F_\delta$ and $\mathcal F_\delta$ for $\tilde{\mathcal F}_\delta''$, $\tilde{\mathcal F}_\delta'=\tilde F_\delta$ and $\tilde{\mathcal F}_\delta$, respectively. }

{\bf (3)} One defines a $(1,1)$-tensor field  $\bf{G}_\delta$ on $B_\delta(Y_0)\subset M_0$ as follows.  ${\bf G}_\delta|_p$, $p\in B_\delta(Y_0)$, is the self adjoint endomorphism of $T_pM$ given by
$${\bf G}_\delta={ \bf I} + 2 F_\delta(x_s^n) {\bf L}- 2C\mathcal F_\delta(x_s^n){\bf P}^T.$$
${\bf P}^\top: TM_0|_{Y_0} \rightarrow TY_0$ is the tangential projection operator.  Note that for $p\in B_\delta(Y_0)$  the operators ${\bf L}_{p}$ and ${\bf P}^{\top}|_p$  are obtained  via parallel transport of ${\bf L}_{q}$ along the unique  normal geodesic from the point $q\in Y_0$ with $d_{Y_0}(p)= d(q,p)$ to 
$p$. Moreover $F_\delta, \mathcal F_\delta\rightarrow 0$ uniformly as $\delta \downarrow 0$. 
The tensor field ${\bf G}_\delta|_p$ coincides with the identity map $\bf I$ on $T_pM_0$ outside of $B_\delta(Y_0)$ and on $Y_0$,  and converges uniformly to $\bf I$ if $\delta \rightarrow 0$. 
\smallskip

Then, one defines the metric $g_\delta$ on $M_0$ via 
$$ g_\delta=\begin{cases} g_0(\cdot , {\bf G}_\delta \cdot) &\mbox{ on } B_\delta(Y)\\
g_0 & \mbox{ otherwise}.
\end{cases}$$
In local coordinates $(x^1, \dots, x^n)$ for $B_\delta(Y)$ this is
$$
(g_\delta)_{\alpha\beta}= (g_0)_{\alpha\beta} + 2 F_\delta(x^n) ({\bf L})_{\alpha\beta} - 2C \mathcal F_\delta(x^n) ({\bf P}^{\top})_{\alpha\beta}
$$
where $({\bf L})_{\alpha\beta}=\langle \frac{\partial}{\partial x^\alpha}, {\bf L} \frac{\partial}{\partial x^\beta}\rangle$ and $({\bf P}^\top)_{\alpha\beta}=\langle \frac{\partial}{\partial x^\alpha}, {\bf P}^\top \frac{\partial}{\partial x^\beta}\rangle.$
\smallskip

It is easy to check that $g_\delta$ is a smooth Riemannian metric on $M_0$ (that coincides with $g_0$ outside of $B_\delta(Y_0)$ and $Y_0$) provided $\delta>0$ is sufficiently small. Moreover, the level sets of $d_{Y_0}$ in $M_0$ w.r.t. $g_\delta$ coincide with the level sets w.r.t. $g_0$ since ${\bf L}\frac{\partial}{\partial x^n}= 0 = {\bf P}^{\top} \frac{\partial}{\partial x^n}$. 
Then, one defines the $C^0$ metric $g_{(\delta)}$ via 
$$g_{(\delta)}= \begin{cases} g_\delta & \mbox{ on } M_0, \\
g_1 & \mbox{ on }M_1. 
\end{cases}$$

\red{
In fact, $g_{(\delta)}$ is a $C^1$-metric, since the transversial derivatives of the coefficients of  $g_\delta$ and of $g_1$ at  boundary  points coincide. We can check this directly from the definition of $g_\delta$, $F_\delta$ and $\mathcal F_\delta$. 
Since away from the boundary $g_{(\delta)}$ is smooth,  coefficients of the metric $g_{(\delta)}$ are in $W^{2, \infty}_{loc}$. }
\smallskip

{\bf (4)} Schlichting then computes  the Ricci curvature  \begin{align}\label{trace} \ric_{g_\delta}=\sum_{i=1}^n (g_\delta)_{\alpha\beta} \mathcal R_\delta(\cdot, dx^\alpha, \cdot, dx^\beta)\end{align} of $g_\delta$ where $\mathcal R_\delta$ is the curvature tensor of $g_\delta$.  
\smallskip

The main estimate for $\ric_{g_\delta}$ is 
\begin{align}\label{est}\mbox{Ric}_{g_\delta}(\mathcal R_\delta)\geq &\mbox{Ric}_g(\mathcal R)- f^2_\delta\mbox{Ric}_g(\mathcal A)+ f_\delta \mbox{Ric}_g(\mathcal B)\nonumber\\
&- 2f_\delta'\mbox{Ric}_{g_\delta}(\mathcal L)+ 2 f_\delta^2\mbox{Ric}_g (\mathcal L^2) + 2 Cf_\delta \mbox{ Ric}_g(\hat{\mathcal I}) - \epsilon(\delta) \mbox{id}_{TM_0}\end{align}
where $\mathcal A, \mathcal B, \mathcal L$, $\mathcal L^2$ and $\hat{\mathcal I}$ are endomorphisms of $\Lambda^2 TM$ of the form $C\wedge D$ for self-adjoint endomorphisms $C,D$ of $TM$.  The wedge product $\wedge$ between $C$ and $D$ is defined via $$(C\wedge D)(v\wedge w) = \frac{1}{2} C(v) D(w) - C(w) D(v).$$ 
$\mbox{Ric}_h(\cdot)$  is the operator that  takes the trace of such endomorphisms  on $\Lambda^2 TM$ as in \eqref{trace} w.r.t. an inner product $h$.  Most relevant for us is that $\mathcal L= {\bf L} \wedge {\bf P}^\perp$. 

Whenever just $f_\delta$ appears in \eqref{est} but not $f_\delta'$, then the corresponding $\mbox{Ric}$ can be uniformily estimated by an $\epsilon(\delta)$ with $\epsilon(\delta)\downarrow 0$ if $\delta\downarrow 0$. Hence, the inequality \eqref{est} reduces essentially to 
\begin{align}\label{alsoimportant}\ric_{g_\delta}=\mbox{Ric}_{g_\delta}(\mathcal R_\delta)\geq (\kappa- \epsilon(\delta)) \mbox{id}_{TM_0} - 2 f'_\delta \mbox{Ric}_{g_\delta}(\mathcal L).\end{align}
Now $f'_\delta$ is negative on $[0,\delta^4)$, does not exceed $\epsilon(\delta)$ on $[\delta^4, \delta]$ and vanishes everywhere else. 

It is enough to show that $\mbox{Ric}_{g_\delta}(\mathcal L)$ is non-negative. 
\begin{jjj}
More precisely: If $\mbox{Ric}_{g_\delta}(\mathcal L)\geq 0$ and  $f_\delta'<0$, we can drop $-2f'_\delta \mbox{Ric}_{g_\delta}(\mathcal L)$ in \eqref{alsoimportant}.  Otherwise $f'_\delta\leq C\delta$. Since $Y$ is compact, there is a constant $A$ that uniformly estimates $ \mbox{Ric}_{g_\delta}(\mathcal L)$. Hence,  $-2f'_\delta\mbox{Ric}_{g_\delta}(\mathcal L)$ is uniformly estimated from above by  $\epsilon(\delta)$ with $\epsilon(\delta)\downarrow 0$ for $\delta\downarrow 0$.
\end{jjj}

  For this we fix $x\in M_0$ close to $Y$ such that it is inside of a local cooardinate chart $(x^1, \dots, x^n)=\phi:= \phi_s, s=1, \dots, N$, of the form we considered before.  Moreover one can construct  $x^1, \dots, x^{n-1}$ such that the coordinate vector fields $\frac{\partial}{\partial x^i}$ are orthonormal in $x$ w.r.t. $g_x$ and ${\bf L}_x$ is diagonal w.r.t. $\frac{\partial}{\partial x^i}\big|_x$.  By construction of $g_\delta$ this implies that $g_\delta|_x$ is diagonal and the local coefficients in $x$ are $(g_\delta)_{ij}(x)= \mu_i \delta_{ij}$ with $\mu_i>0$. Then it follows for a vector $\xi\in T_xM_0$ that
{\color{black}
\begin{align}\label{id:ricci}
\mbox{Ric}_{g_\delta}(\mathcal L)(\xi, \xi)= \frac{1}{2} \frac{1}{\mu_n} L(\xi, \xi) + \frac{1}{2} (\xi^n)^2 \sum_{l=1}^n \frac{1}{\mu_l} ({ L})_{ll}.
\end{align}
The first term is the (parallel transported) 2nd fundamental form at $\xi$ times $\frac{1}{2}\frac{1}{\mu_n}$. In the second term we have the terms $L_{ll}= L\left(\frac{\partial}{\partial x^l}|_x, \frac{\partial}{\partial x^l}|_x\right)\geq 0$ times $\frac{1}{2} (\xi^n)^2= \frac{1}{2} g_x\left(\frac{\partial}{\partial x^n}, \xi\right)^2.$
It follows  that $\mbox{Ric}_{g_\delta}(\mathcal L)(\xi, \xi)$ is non-negative by assumption.  

We also notice that in the second term $\sum_{l=1}^n \frac{1}{\mu_l} ({ L})_{ll}$ is the trace $\mbox{tr}^{g_\delta}L(x)$ of $L(x)$ w.r.t. $g_\delta|_x$. }
\smallskip

{\bf (5)}
In the final step $g_{(\delta)}$ is smoothened. For this one covers the manifold $M_0$ with open sets $U_s$, $s=1, \dots, S,$ and we also pick coordinate charts $\phi_s$ as before such that $U_s\subset U'_s$.  Since  $Y_0\simeq Y_1$ is compact, it is covered by finitely many $U$, again these are $U_1, \dots, U_N$.  In each $U_i$ the metric $g_{(\delta)}$ has  a local representation given by functions $(g^{s}_{(\delta)})_{\alpha\beta}$, $\alpha, \beta=1, \dots, n$. Each $(g_{(\delta)}^s)_{\alpha\beta}$ is  mollified with a   standard mollifier function $\rho\in C^\infty_c(B_1(0))$ with $\int_{\R^{\lceil N \rceil}} \rho =1$.  
More precisely, one defines
$$
(g^{s, h}_{(\delta)})_{\alpha\beta}= \int_{B_h(0)} \rho_h(z)(g^{s}_{(\delta)})_{\alpha\beta}(x-hz) d\mathcal L^n(z).
$$
 Finally, with help of a partition of unity one builds a smooth Riemannian metric $g^h_{(\delta)}$ on $M$.  Then $g^h_{(\delta)}$ is the desired Riemannian metric $g^\delta$. 
\smallskip

\subsection{\bf Glued spaces with weights}\label{subsec:weightedgluing}
Let $\Phi_i:M_i\rightarrow \R$, $i=0,1$, be smooth such that $\Phi_0|_{Y_0}=\Phi_1|_{Y_1}$. Consider $M_i$ as weighted Riemannian manifolds $(M_i, \Phi_i)$, $i=0,1$. 
Define  $\Phi= \Phi_0\cup \Phi_1: M_0\cup_{\mathcal I}M_1\rightarrow \R$ as 
$$\Phi(x)=\begin{cases} \Phi_0(x) & x\in M_0\\
\Phi_1(x) & \mbox{otherwise}.
\end{cases}$$
Let $\nu_i: \partial M_i\rightarrow TM_i$ be the inward unit normal vector field. 
\begin{assumption}\label{ass:3} Let $\Phi_i, i=0,1$, as before. We assume $\Phi_0(p)= \Phi_1(p)>0$  for all $p\in Y_0\simeq Y_1$ and
\begin{enumerate}
\item  
${\displaystyle \sum_{i=0,1} \langle \nabla \log \Phi_i, \nu_i \rangle =: H \leq \mbox{tr}\Pi\mbox{ on } Y_0\simeq Y_1}\mbox{ 
where $\Pi= \Pi_0|_{Y_0}+ \Pi_1|_{Y_1}$.} $
\end{enumerate}
In particular, since $\Pi\geq 0$ on $Y$, the assumption is satisfied if
\smallskip
\begin{enumerate}
\item[(2)]  ${\displaystyle \sum_{i=0,1} \langle \nabla \log \Phi_i, \nu_i \rangle \leq 0\mbox{ on } Y_0\simeq Y_1.}$
\end{enumerate}
\end{assumption}
\begin{example}
Let $[a,b]$ and $[b,c]$ be intervals in $\R$ and let $\Phi_0\in C^{\infty}([a,b])$ and $\Phi_1\in C^{\infty}([b,c])$ such that $\Phi_0, \Phi_1\geq 0$ and  $([a,b], \Phi_0)$ and $([b,c], \Phi_1)$ have a Bakry-Emery $N$-Ricci curvature lower bound $K$. If 
$$
\frac{d^-}{dt} \Phi_0(c) + \frac{d^+}{dt}{\Phi_1}(c)\leq 0
$$
then the function $\Phi$ on $[a,c]$ is $(K,N)$-concave and therefore $([a,c], \Phi d\mathcal L^1)$ satisfies the curvature-dimension condition $CD(K,N)$. 
\end{example}

\section{Proof of Theorem 1.1 and Theorem 1.2}\label{sec:4}
\subsection{Warped products over Riemann manifolds} We recall some facts about Riemannian warped products \cite{oneillsemi, ketterer}.
\smallskip

Let $M$ be a Riemannian manifold with boundary  $\partial M$ and let $f: M\rightarrow [0,\infty)$ be smooth. We set $$\mbox{$\mathring M= M\backslash f^{-1}(\{0\})$ and $\mathring f:= f|_{\mathring M}$.}$$  

Let $\Pi_{\partial M}$ the second fundamental form of $\partial M$. We set $\partial M\cap f^{-1}(\{0\})=:X$. $X$ is a connected component of $\partial M$. Let $F$ be a closed Riemannian manifold.  We define $\mathring C= \mathring M\times F$. 
\smallskip

The Riemannian warped product 
$
 \mathring{M}\times_{\mathring{f}} F
$ with respect to $\mathring{f}$ 
is the product space $\mathring C=\mathring M\times F$ with the Riemannian metric $\tilde g$  given by
\begin{displaymath}
 \tilde g:=(\pi_{\sM})^*g_{\sM}|_{\mathring M}+(f\circ\pi_{\sM})^2(\pi_F)^*g_F.
\end{displaymath}

Here $g_{\sM}$ and $g_F$ are the Riemannian metrics of ${M}$ and $F$ respectively.
The length of a Lipschitz-continuous curve $\gamma=(\alpha,\beta)$ in $\mathring M\times_{\mathring f} F$ is
\begin{displaymath}
\L_{\tilde g}(\gamma)=\int_0^1\sqrt{g_{\sM}({\alpha'}(t),{\alpha'}(t))+f^2\circ \alpha(t) g_F({\beta'}(t),{\beta'}(t))}dt.
\end{displaymath}
The  distance on $\mathring M\times_{\mathring f} F$ is defined by
$$
 |(p,x),(q,y)|=\inf {\L}_{\tilde g}(\gamma)
$$
where the infimum runs over all Lipschitz curves between $(p,x)$ and $(q,y)$ in $\mathring M\times F$. 
\smallskip

The distance $|\cdot, \cdot|$ is also defined on $M\times F$ by the same infimum w.r.t. to Lipschitz curves in $M\times F$.
\smallskip

The metric warped product $M\times_f F$ between $M$ and $F$ w.r.t. $f$ is  then defined as the quotient space 
$$(M\times F/\sim, ||\cdot, \cdot||) \ \ \mbox{ where } (p,x)\sim (q,y) \ \Leftrightarrow \ |(p,x), (q,y)|=0$$
and the distance $\left\|\cdot, \cdot\right\|$ is defined via $\left\|[p,x], [q,y]\right\|= |(p,x), (q,y)|$.  

The smooth space $\mathring M\times_{\mathring f} F$ locally embeds into $M\times_f F$. 

\begin{proposition}\label{bigformula}
Let $M^n$ and $F^{d}$ be Riemannian manifolds and let ${f}:{M}\rightarrow [0,\infty)$ be smooth. 
Consider $\xi+v\in T(\mathring{M}\times F)_{(p,x)}=TM_p\oplus TF_x$. It holds
\begin{align*}
\ric_{\mathring M\times_{\mathring f}F}(\xi+v, \xi +v)=&\ric_{M}(\xi, \xi)-d\frac{\nabla^2f(\xi, \xi)}{f(p)}\\
&+\ric_{\sF}(v,v)-\left(\frac{\Delta f(p)}{f(p)}+(d-1)\frac{|\nabla f(p)|^2}{f^2(p)}\right)|v|_{\tilde g}^2.
\end{align*}
\end{proposition}

We fix a parameter $N\in (0, \infty)$. The metric measure space  given by the metric warped product $M\times_fF$ equipped with the measure $$f^{N} \vol_M \otimes \vol_F$$ is denoted with $M\times_f^{N} F$. Analogously $\mathring M\times_{\mathring f}^N F$ is $\mathring M\times_{\mathring f} F$ equipped with $(\mathring f)^N \vol_M \otimes \vol_F$. If $N= \dim F= d$, then $(\mathring f)^d \vol_M\otimes \vol F$ is exactly the Riemannian volume of $\mathring M \times_f F$. 

 Let $(M,g_M)$ be a Riemannian manifold that  is also an Alexandrov space with curvature bounded from below by $\kappa$. This is equivalent to  sectional curvature  bounded from below by $\kappa$ and $\Pi_{\partial M}\geq 0$.
The main theorem of \cite{ketterer1} states the following.
\begin{theorem}\label{th:ketterer1} If $\ric_{g_F}\geq (N-n-1) K_F>0$ and $f:M\rightarrow [0,\infty)$ is $\mathcal F\kappa$-concave and satisfies the condition $(\dagger)$ as well as
\begin{align}\label{condition}|\nabla f|\leq \sqrt{K_F} \mbox{ on } \partial M\cap f^{-1}(0)\end{align}
then the mm space $M\times_f^{N-n} F$ satisfies the  condition $CD((N-1)k, N)$. 
\smallskip\\
$(\dagger)$ If $M^{\dagger}$ is the result of gluing two copies of $M$ together along $\partial M\backslash X$ and $f^{\dagger}: M^{\dagger}\rightarrow [0,\infty)$ is the tautological extension of $f$, then  $f^\dagger$ is $\mathcal F\theta$-concave.
\end{theorem}
A sufficient and necessary criteria for the condition \eqref{condition} is  given by the following proposition {\cite[Proposition 3.1]{albi}}.
\begin{proposition} \label{prop:albi} Let $M$ be an Alexandrov space. 
Let  $f: M\rightarrow [0,\infty)$ is $\mathcal F k$-concave.  We set $X= \partial M\cap f^{-1}(\{0\})$ and $f$ satisfies the condition $(\dagger)$ above. Then the following statements are equivalent. 
\begin{enumerate}\smallskip
\item $L\geq kf^2$ if $X=\emptyset$, or,   
$L\geq 0$ and $|\nabla f|_p^2\leq  L$ on $X$, if $X\neq \emptyset$.
\medskip
\item $L\geq k f^2$ and $|\nabla f|^2+ k f^2 \leq L^2$ on $M$.
\end{enumerate}
%
%
\end{proposition}
\smallskip

Now let $(M, g, \Phi)$ be a weighted Riemannian manifold with boundary such that $\overline \kappa\geq \mbox{sec}_M\geq \underline \kappa$ for $\underline \kappa<0< \overline \kappa$ and $$\ric_{g}^{\Phi,N}\geq K=:(N-1)\eta$$ with $\eta:=  \frac{K}{N-1}$.  
\smallskip

The function  $\Phi^{\frac{1}{N-n}}$ is  $\mathcal F \theta$-concave with $\theta:=\min\{0,-\overline \kappa, \underline \kappa, \eta\}$ because of Fact \ref{firstfact}.  Moreover,  $\ric_g^{\Phi, N} \geq (N-1) \theta$ and $\mbox{sec}_M\geq \theta$. By Corollary \ref{cor:pos} we have $\Phi>0$ on $M\backslash \partial M$. We assume $X= \Phi^{-1}(\{0\})\subset \partial M$ is a connected component of $\partial M$.
\smallskip

\begin{assumption} We assume that 
$$
\sup_M |\nabla \Phi|<\infty.
$$
In particular, there exists $L>0$ such that
\begin{align}\label{idd}
\mbox{ $L\geq \theta \Phi^\frac{2}{N-n}$ \ \& \  $|\nabla \Phi^{\frac{1}{N-n}}|^2+\theta \Phi^\frac{2}{N-n}\leq L$.}
\end{align}
We note that first inequality holds trivially.

Since $\Phi\in C^{\infty}(M)$, this is satisfied, for instance, if $M$ is compact.
\end{assumption}
\noindent

We set $$\cN:=\min\{n\in \N: n\geq N\}$$
but in  the following  we assume by abuse of notation that $\lceil N \rceil=N\in \N$. 
We  also set $\Phi^{\frac{1}{N-n}}=f$. 

We will use an idea introduced by Lott in \cite{lobaem} to study the weighted Riemannian manifold $(M, g, \Phi)$.

\begin{proposition}\label{prop:warped}
Let $(F, g_F)$ be a compact $(N-n)$-dimensional Riemannian manifold with $$\ric_F\geq  (N-n-1) L- (N-1)\theta\max_{p \in M}   f(p) +  (N-1)\eta=:(N-n-1)\tilde L.$$
Then the warped product $ \mathring M\times_{\mathring f} F= (\mathring C, \tilde g)$ satisfies $\ric_{\mathring M\times_{\mathring f} F}\geq (N-1)\eta$.
\end{proposition}
\begin{proof}
Let  $\xi+v\in TM_p\oplus TF_x$. Then we have by Proposition \ref{bigformula}
\begin{align*}
\ric_{M\times_f F}(\xi+v, \xi+v)&=\ric_{M}(\xi, \xi)-(N-n)\frac{\nabla^2f(\xi, \xi)}{f(p)}\\
+&\ric_{\sF}(v,v)-\left(\frac{\Delta f(p)}{f(p)}+(N-n-1)\frac{(\nabla f_p)^2}{f^2(p)}\right)|v|_{\tilde g}^2.
\end{align*}
Since $f= \Phi^{\frac{1}{N-n}}$ we recognize  that 
$$\ric_{M}(\xi, \xi)-(N-n)\frac{\nabla^2f(\xi, \xi)}{f(p)}=\ric^{\Phi,N}(\xi, \xi)\geq (N-1)\eta |\xi|_{g_i}^2.$$
Moreover, because of \eqref{idd} and since $f$ is $\mathcal F\theta$-concave, it follows
$$\frac{\Delta f(p)}{f(p)}+(N-n-1)\frac{(\nabla f_p)^2}{f^2(p)}\leq  - n\theta +(N-n-1) \left(\frac{ L}{f^2(p)}- \theta\right).$$
Combining these formulas together with $\ric_F\geq (N-n-1) \tilde L$  yields
\begin{align*}&\!\!\!\!\!\!\!\!\ric_{M\times_{f}F}(\xi+v, \xi+v)\geq (N-1)\eta |\xi|^2_{g_i} \\
&+ \ric_F(v,v) - (N-n-1)\frac{L}{f^2(p)} f^2(p)|v|^2_{g_F}+ (N-1)\theta f^2(p) |v|^2_{g_F}\\
=& (N-1)\eta(|\xi|_{g_i}^2 + |v|_{g_F}^2).
\end{align*}
This is the claim.
\end{proof}
In the following we choose $F=r\mathbb S^{N-n}$ with $r>0$ such that $\frac{1}{ r^2} \geq  \tilde L.$
We will often omit the dependency on $r$ and write $M\times_f F$ for $M\times_{f} r\mathbb S^{N-n}$. 
%
\smallskip

In general the space $M\times_fF$ is not a smooth manifold in points where the metric $g$ degenerates. These are the points $(p,x)\in M\times_fF$ with $p\in X$ where $f(p)=0$. 
%
%

\begin{fact}
The boundary $\partial \mathring C$ of $\mathring C$ is $\partial \mathring{M}\times F$. 
\end{fact}

\begin{proposition} \label{prop:connection} Let $(M, g, \Phi)$ be a weighted Riemannian manifold with boundary. Consider the Riemannian warped product $\mathring M\times_{\mathring f} F=(\mathring C, \tilde g)$ for $f=\Phi^{\frac{1}{N-n}}$ and $F$ as before. Let $p\in \partial M\backslash f^{-1}(0)= \partial \mathring M$ and $x\in F$.
It holds 
$$\tilde \Pi (\xi, v, \chi, w)= \Pi_{\partial M}(\xi, \chi)- g(\nu|_p, \nabla \log f|_p) f^2(p) g_F(  v,  w)$$
for $\xi, \chi\in T_p\partial \mathring M, \ v, w\in T_xF$
where $\nu$ is the inward unit normal vector field along $\partial \mathring M$ and $\tilde \Pi$ is the second fundamental form of $\partial \mathring C$. 
\end{proposition}
\begin{proof}
We consider vector fields $\mathcal X,\mathcal Y$ on $\mathring M=M\backslash f^{-1}(\{0\})$ and $V,W$ on $F$ such that $\mathcal X_p=\xi, \mathcal Y_p=\chi \in T_p\partial M$ and $V_x=v, W_x=w$. Let  $\tilde{\mathcal X}, \tilde{\mathcal Y}, \tilde{V}, \tilde{W}$ be the horizontal and vertical lifts, respectively. 
$\tilde{\nabla}$ denotes the Levi-Civita-connection of {$\mathring M\times_{\mathring{f}} F$}.  
Then the following identities holds
\begin{itemize}
\item[(1)] $\tilde{\nabla}_{\tilde{X}}\tilde{Y}=\widetilde{\nabla^{\mathring{M}}_{X}Y}$,
\smallskip
\item[(2)] $\tilde{\nabla}_{\tilde{X}}\tilde{V}=\tilde{\nabla}_{\tilde{V}}\tilde{X}=\left(\frac{Xf}{f}\circ \pi_{\sM}\right)\tilde{V}$,
\smallskip
\item[(3)] $\tilde{\nabla}_{\tilde{V}}\tilde{W}=-\left(\frac{\tilde g( \tilde{V},\tilde{W})}{f}\circ \pi_{\sM}\right)\widetilde{\nabla f}+\widetilde{\nabla^F_V W}$.
\end{itemize}
For $(p,x)\in  \partial \mathring M\times F$, we have $T_{(p,x)}\partial \mathring C= T_p\partial \mathring M\oplus T_x F$ as well as $N_{(p,x)} \partial \mathring C= N_{p} \partial \mathring M \oplus \{0\}$ where $N\partial \mathring M$ is the normal vector bundle of $M$. Thus, if $\nu$ is the inward unit normal vector field along $\partial \mathring M$, then $(\nu, 0)$ is the inward unit normal along $\partial \mathring C$. 
We compute:
\begin{align*}
\tilde \Pi_{(p,x)}\left({\mathcal X}_p,  V_x, {\mathcal Y}_p,  W_x\right)=& \tilde g\left((\nu, 0),\tilde \nabla_{(\tilde{\mathcal X}+ \tilde V)}( \tilde{\mathcal Y}+\tilde W)\big|_{(p,x)}\right)\\
=&\tilde g\left( (\nu,0), \widetilde{\nabla^{\mathring M}_{\mathcal X}\mathcal Y}\big|_{(p,x)}\right) + \tilde g\left( (\nu,0), \tilde{\nabla}_{\tilde{\mathcal X}}\tilde W\big|_{(p,x)}\right)\\
& +\tilde g\left((\nu,0), \tilde{\nabla}_{\tilde V}\tilde{\mathcal Y}\big|_{(p,x)}\right)+ \tilde g\left( (\nu,0), \tilde \nabla_{\tilde V}\tilde W\big|_{(x,p)}\right)\\
=&g\left( \nu, \nabla^{\mathring M}_{\mathcal X}\mathcal Y\big|_p\right)- \frac{\tilde g( \tilde V_{(p,x)}, \tilde W_{(p,x)})}{f(p)} g\left( \nu, \nabla f\big|_p\right)\\
=&\Pi_p(\mathcal X_p, \mathcal Y_p)-\underbrace{ {\tilde g( \tilde V_{(p,x)}, \tilde W_{(p,x)}) } }_{f^2(p) g_F(V_x, W_x)}g\left( \nu_i, \nabla \log f\big|_p\right).
\end{align*}
This is the claim. 
\end{proof}
\begin{remark}\label{rem:2sided}
If $\Pi\geq 0$ and $g(\nu, \nabla \log f)\leq 0$ on $\partial M\backslash X$, then \begin{center}$\tilde \Pi|_{(p,x)}\geq 0$ $\forall x\in F$ and $\forall p\in \partial M\backslash X$.\end{center}

Since $f\in C^\infty(M)$ and  $\partial M$ is compact, we always have that $|g(\nu, \nabla \log f)|$ is uniformily  bounded on $\partial M\backslash X$. Hence 
$$\big\| \tilde \Pi\big\| \leq A<\infty \mbox{ on } \partial C\backslash X\times F.$$
\end{remark}

\begin{corollary} Let $(M,g, \Phi)$ and $F=\mathbb S^{N-n}$ be as before. (1)
Then for every point $(p,x)$ with $p\in (M\backslash \partial M) \cup X$ in the metric warped product $M\times^{N-n}_fF$ there exists a neighborhood $U_{(p,x)}$ that satisfies the local curvature-dimension $CD_{loc}(K, N)$.  (2) If $\Pi_{\partial M}\geq 0$ and $g(\nu, \nabla \log f)\leq 0$ on $\partial M\backslash X$, then $M\times^{N-n}_f F$ satisfies the condition $CD(K,N)$. 
\end{corollary}
\begin{proof} We prove the second claim.
We observe  that, if $\Pi_{\partial M}\geq 0$, then $(M,g_M)$ is an Alexandrov space with curvature bounded from below by $\theta$. Moreover $f$ is $\mathcal F\theta$-concave and satisfies $|\nabla f|^2 + \theta f^2 \leq L$.  The boundary condition $g(\nu, \nabla \log f)\leq 0$ on $\partial M\backslash X$ implies the condition $(\dagger)$ above. We can apply Theorem \ref{th:ketterer1} and  hence $M\times^{N-n}_f F$ satisfies the condition $CD((N-1)\theta, N)$.  In particular $\tilde \Pi\geq 0$ on $(\partial M\backslash X) \times F$. On the other hand by   Proposition \ref{prop:connection} the Ricci tensor of $\mathring M\times_{\mathring f}^{N-n} F$ is bounded from below by $K$. Hence, following standard arguments about the characterization of lower Ricci curvature bounds via entropy convexity \cite{cms, stugeo2, kettererlp}, when an $L^2$-Wasserstein geodesic $(\mu_t)_{t\in [0,1]}$ between absolutely continuous measures $\mu_0$ and $\mu_1$ is concentrated in $\mathring M\times_{\mathring f}^{N-n} F$, then  the inequality \eqref{ineq:cd} for the $N$-Renyi entropy holds along $(\mu_t)_{t\in [0,1]}$. Together with  Theorem 3.4 in \cite{ketterer1} {\color{black} (for which we require the condition $CD((N-1)\theta, N)$ first)} this yields that $M\times_f^{N-n} F$ satisfies the condition $CD(K,N)$. 
\smallskip
\\
The first claim follows similarly.  For this we notice that we can localize the previous arguments on neighborhoods $U_{(p,x)}$ of points $(p,x)$ in the metric warped product with $p\in (M\backslash \partial M)\cup X$. 
\end{proof}

\subsection{Gluing of warped products} Now we apply the results about warped products to our gluing construction. 
\smallskip

Let $(M_i, g_i,\Phi_i)_{i=0,1}$ be weighted Riemannian manifolds with boundary that satisfy Assumption \ref{ass:1} and \ref{ass:3}. 
We set $f_i= \Phi_i^{\scriptscriptstyle \frac{1}{N-n}}$ and choose $F=r\mathbb S^{N-n}$  for $r>0$ such that $\ric_{r\mathbb S^{N-n}}\geq \tilde L$ as before. Then we  build the weighted warped products $M_i\times_{f_i}^{N-n} F=C_i$, $i=0,1$. 
The   Riemannian warped product metrics  $\tilde g_i$, $i=0,1$, are defined on $\mathring C_i$. 
\medskip

Set $\tilde Y_i = Y_i \times \mathbb S^{\lceil N \rceil- n}$, $i=0,1$. $\tilde Y_0$ and $\tilde Y_1$ equipped with corresponding restricted metric are isometric via  $\tilde{\mathcal  I}: \tilde Y_0\rightarrow \tilde Y_1$ given by $\tilde{\mathcal I}(p,x)=(\mathcal I(p), x)$.
\smallskip

Hence we can define the metric glued space $\mathring C_0\cup_{\tilde{ \mathcal{I}}} \mathring C_1$. 
The corresponding $C^0$ Riemannian metric is 
$$\tilde g= \begin{cases} \tilde g_0 & \mbox{ on } \mathring C_0\\
\tilde g_1 & \mbox{ on } \mathring C_1.
\end{cases}
$$

We can follow the construction of the metric $g^\delta$ in Subsection \ref{subsec:gdelta}. This yields  a family of smooth Riemannian metric $(\tilde g^\delta)_{\delta>0}$ on $\mathring C_0\cup_{\tilde{\mathcal I}}\mathring C_1$ such that 
\begin{enumerate}
\item $\tilde g^\delta$ coincides with $\tilde g$ outside of $B_\delta(\tilde Y_0 \simeq \tilde Y_1)$, 
\smallskip
\item $\tilde g^\delta$ converges uniformly to $\tilde g$ as $\delta \downarrow 0$. 
\end{enumerate}
\smallskip
{\bf First, we procede assuming (2) in 
 Assumption \ref{ass:3}, i.e. \begin{align}\label{ass3}\mbox{$\sum_{i=0,1} \langle \nabla \log \Phi_i, \nu_i \rangle \leq 0\mbox{ on } Y_0\simeq Y_1$.}\end{align}} By Proposition \ref{prop:connection} it   follows
for $p\in Y_0\simeq Y_1$  that
$$(\tilde \Pi_0 + \tilde \Pi_1)|_{(p,x)}=: \tilde \Pi_{(p,x)}\geq 0 \ \forall x\in F.$$
Hence, we obtain the following corollary.
\begin{corollary}\label{th:gluedwarped} $\ric_{\tilde g^\delta}\geq K -\epsilon(\delta)$ with $\epsilon(\delta)\rightarrow 0$ for $\delta\downarrow 0$.
\end{corollary}
We apply this corollary to deduce a local curvature-dimension condition for $\tilde g$.
We set $\Int (\mathring C_0\cup_{\tilde{\mathcal I}} \mathring C_1) : = (\mathring C_0 \cup_{\tilde{\mathcal I}} \mathring C_1)\backslash \partial (\mathring C_0\cup_{\tilde{\mathcal I}} \mathring C_1)$. 
\begin{corollary}\label{cor:assum}
For every $(p,x)\in \Int(\mathring C_0\cup_{\tilde{\mathcal I}}\mathring C_1)$ there is $R>0$ such that  $B^{\tilde g}_{2R}((p,x))\subset \Int (\mathring C_0 \cup_{\tilde{\mathcal I}}\mathring C_1)$ and $B^{\tilde g}_R((p,x))$ satisfies the condition $CD_{loc}(K,N)$. 
\end{corollary}
\begin{proof}
We choose $R>0$ such that $B^{\tilde{g}}_{4R}(x)\subset \Int(\mathring C_0\cup_{\tilde{\mathcal I}}\mathring C_1)$. The uniform convergence of $\tilde g^\delta$ to $\tilde g$ for $\delta\downarrow 0$ implies  that $B^{\tilde g^\delta}_{3R}((p,x))\subset \Int (\mathring C_0 \cup_{\tilde{\mathcal I}}\mathring C_1)$ for  all $\delta>0$ sufficiently small,  and Gromov-Hausdorff convergence of $\overline B^{\tilde g^\delta}_{2R}(x)$ to $\overline B^{\tilde g}_{2R}(x)$ for $\delta\downarrow 0$. Since $\ric_{\tilde g^\delta}\geq K -\epsilon(\delta)$, we have that $\overline B^{\tilde g^\delta}_{R}(x)$  satisfies the condition $CD_{loc}(K- \epsilon(\delta), N)$. 
Hence, by stability of the curvature-dimension condition w.r.t. measured Gromov-Hausdorff convergence it follows that $B_R^{\tilde g}(x)$ satisfies the condition $CD_{loc}(K, N)$. 
\end{proof}
\begin{remark}\label{rem:simple}
If we assume \eqref{ass3}, we can finish the proof of Theorem \ref{main2} as follows.

We first prove a local convergence result.
By the construction of $\tilde g$  it is clear that for every $(p,x)\in \Int(\mathring C_0\cup_{\tilde{\mathcal I}}\mathring C_1)$ and for $R>0$ such that $\overline B_{2R}^g(p)\subset \Int (\mathring M_0\cup_{\mathcal I} \mathring M_1)$ the sequence of closed sets $\overline B_{2R}^{ g}(p)\times r \mathbb S^{N-n}$ equipped with  the distance function induced by $\tilde g$ converges  in Gromov-Hausdorff sense to $\overline B_{2R}^{g}(p)$  as $r\downarrow 0$. For $r>0$ an $\epsilon(r)$-GH-approximation (with $\epsilon(r)\downarrow 0$ for $r\downarrow 0$) is given by the projection of $\overline B_{2R}^{ g}(p)\times r \mathbb S^{N-n}$  onto $\overline B_{2R}^{g}(p)$. 

\begin{jjj}
$B_R^g(p)\times r\mathbb S^{N-n}$, as an open subset in $\Int(\mathring C_0 \cup_{\tilde{\mathcal I}}\mathring C_1)$ satisfies the condition $CD_{loc}(K-\epsilon(\delta), N)$ for $R>0$ sufficiently small.
\end{jjj}

To see weak convergence of the measure we note  
that the volume form of $\tilde g$ is given by $$\vol_{\tilde g}= (f r)^{\lceil N \rceil-n} \vol_{g} \vol_{\mathbb{S}^{\cN-n}}.$$
The normalized volume form  $$\tilde \m=\frac{1}{\omega_{\cN-n}}{r^{-\lceil N \rceil+n}}\vol_{\tilde g}= f^{N-n} \vol_g \frac{1}{\omega_{\lceil N\rceil-n}} \vol_{\mathbb S^{\cN-n}}$$ is therefore independent of $r$ where $\omega_n$ is the volume of the $n$-dimensional standard sphere. The pushforward of $\tilde m$ under the projection is exactly $\Phi\vol_g= \m$. Hence $\tilde \m|_{\overline B_{2R}^{ g}(p)\times r \mathbb S^{N-n}}$ converges weakly to $\m|_{\overline B_{2R}^g(p)}$. 

In the statement of Theorem \ref{main2} we assume $X_i=\emptyset$ for $i=0,1$. Hence $C_i= \mathring C_i$, $i=0,1$, and $C_0\cup_{\tilde{\mathcal I}}C_1$ is a  Riemannian manifold with boundary such that we have a two-sided uniform bound for the second fundamental form of $\partial (C_0\cup_{\tilde{\mathcal I}}C_1)$ that is independent of $r>0$ (Remark \ref{rem:2sided}). Therefore, by \cite{wong}, it follows that a subsequence converges in \red{pointed} GH sense to a  metric space. On the other hand  we have that $\bar B_{2R}^g(p)\times r\mathbb S^{N-n}$ equipped with $d_{\tilde g}$ converges in GH sense to $\bar B_{2R}^g(p)\subset M_0\cup_{\mathcal I} M_1$ as $r\downarrow 0$. Hence, the  limit in \cite{wong} coincides with $M_0\cup_{\mathcal I}M_1$.  \qed
\end{remark}
\subsection{}
{\bf In the following we will prove Theorem \ref{main2} and Theorem \ref{main3} without assuming (2) in Assumption \ref{ass:3}.}

We consider again the metric $\tilde g^\delta$ as in the previous section that is constructed via the steps  {\bf (1)}, {\bf (2)} and {\bf (3)} in the proof of Theorem \ref{th:glue2}.
We write $g_{\mathbb S^{\lceil N\rceil-n}}=: h$ where $g_{\mathbb S^{\lceil N\rceil-n}}$ is the round metric   on $\mathbb S^{\lceil N \rceil- n}$. 
\begin{lemma}There is a smooth $(2,0)$-tensor  $\tilde h^\delta$ on $(\mathring M_0\cup_{\mathcal I} \mathring M_1)\times \mathbb S^{\uN-n}$ such that 
$$
\tilde g^\delta =  g^\delta + r^2 \tilde h^\delta.
$$
\end{lemma}
\begin{proof}
We introduce local coordinates $\psi_s$ of the form $\psi_s(p,x)= ( \phi_s(p), \check \phi_s(x))$  on $\mathring C_0$ where $\phi_s$ are local coordinates for $M_0$ as in the proof of Theorem \ref{th:glue2} and $\check \phi_s$ are local coordinates on $\mathbb S^{\uN-n}$.  Recall the definition of the operators ${\bf L}$ and ${\bf P}^\top$ from the step {\bf (1)} in the proof of Theorem \ref{th:glue2} but with $\tilde g_0$, $\tilde Y_0$ and  $\tilde \Pi$ in place of $g_0$, $Y_0$ and $\Pi$, respectively. 
Since $\tilde g_0$ is a warped product Riemannian metric, the   operators $\tilde{\bf L}$ and $\tilde {\bf P}^\top$ that appear in the proof of Theorem \ref{th:glue2}  have a block structure when they are represented in local coordinates given by such a chart $\psi_s$.  More precisely, for $\tilde{\bf L}$ this block structure follows from Proposition \ref{prop:connection}:
$$
\begin{pmatrix} ({\tilde {\bf L}}_{\alpha\beta})_{\alpha, \beta=1, \dots, n}& 0\\
0&  ({\tilde { \bf L}}_{\alpha\beta})_{\alpha, \beta=n+1, \dots, \lceil N \rceil}  
\end{pmatrix}
$$
where $\tilde g_0(\cdot, \tilde{\bf L} \cdot) = \tilde \Pi$ with $\tilde \Pi_0|_{\tilde Y_0}+ \tilde \Pi_1|_{\tilde Y_1}=: \tilde \Pi$, and $\tilde \Pi_i$,  $i=0,1$, is the second fundamental form of $\partial C_i$. Moreover
$$\tilde {\bf L}_{\alpha\beta}= g_0\left( \frac{\partial }{\partial x^\alpha}, {\bf L} \frac{\partial}{\partial x^\beta}\right)={\bf L}_{\alpha\beta}$$  for $\alpha, \beta =1, \dots, n-1$, $\tilde{\bf L}_{\alpha n} = \tilde{\bf L}_{\beta n} =0$ for $\alpha, \beta =1, \dots, n$, and 
$$\tilde{\bf L}_{\alpha\beta}=\frac{(\tilde g_0)_{\alpha\beta}}{f}\Big( g_0\left( \nu_0, \nabla f\right) + g_0\left( \nu_1, \nabla f\right)\Big)$$ for $\alpha, \beta= n+1, \dots, \lceil N \rceil$. 

The operator $\tilde{\bf P}^\top$ has the form $\left(\tilde{\bf P}^\top\right)_{\alpha\beta}=\delta_{\alpha\beta}$ for $\alpha, \beta\neq n$ and $\left(\tilde{\bf P}^\top\right)_{\alpha n}= \left(\tilde{\bf P}^\top\right)_{n \beta}=0$ for $\alpha, \beta= 1, \dots, \lceil N\rceil$.

Hence, the metric defined in ${\bf (2)}$ is written  in these local coordinates explicitly as follows:
\begin{enumerate}
\item
$\mbox{for } \alpha, \beta=1, \dots, n$ $$(\tilde g_\delta)_{\alpha\beta}= (g_0)_{\alpha\beta} +2F_\delta(x^n) {\bf L}_{\alpha\beta} - 2C \mathcal F_\delta(x^n){\bf P}_{\alpha\beta}=(g_\delta)_{\alpha\beta}$$
\item $\mbox{for }\alpha, \beta=n+1, \dots, \uN$
$$(\tilde g_\delta)_{\alpha\beta}= r^2 f^2\big(\underbrace{h_{\alpha\beta} + 2 F_\delta(x^n)\tilde{\bf L}_{\alpha\beta} - 2 C\mathcal F_\delta(x^n)\tilde{\bf P}_{\alpha\beta}}_{=: (h_\delta)_{\alpha\beta}}\big)$$

\item for $\alpha\in \{1, \dots, n\} \mbox{ and } \beta\in \{n+1, \dots, \lceil N \rceil\}$
$$(\tilde g_\delta)_{\alpha\beta}= 0.$$
\end{enumerate}
The mollification in step $(\bf{3})$ of the proof of Theorem \ref{th:glue2} preserves this structure and yields
$$\tilde g^\delta= g^\delta + r^2 \tilde h^\delta$$ where the metric $\tilde h^\delta$ is obtained by mollification of  $f^2  h_\delta$ on $M\times \mathbb S^{\lceil N\rceil-n}$. 
\end{proof}
\begin{proposition}\label{prop:gluedwarped2}
We assume Assumption \ref{ass:1},  {\color{black} and  a strict inequality in (1) in Assumption \ref{ass:3}. }

The family of Riemannian metrics $(\tilde g^\delta)_{\delta>0}$ satisfies\begin{align*}\ric_{\tilde{g}^\delta}(\xi+v, \xi+v)&\geq
 (K-\epsilon(\delta))\left( |\xi|_{\tilde g^\delta}^2 + |v|_{\tilde g^\delta}^2\right) + {\color{black} (\min f'_\delta)} \overline H|v|_{\tilde g^\delta}^2
\end{align*}
 for $\xi+v\in T (\mathring C_0\cup_{\tilde{\mathcal I}}\mathring C_1)$
 with $\epsilon(\delta)\rightarrow 0$ for $\delta\downarrow 0$ where  
$$
{\max\left\{ \langle \nu_0, \nabla \log \Phi^\frac{1}{N-n}_0\rangle + \langle \nu_1, \nabla \log \Phi_1^{\frac{1}{N-n}}\rangle \right\}}=: \overline H.
$$
{\color{black} and $f_\delta: [0, \infty) \rightarrow \R$ is the function in the proof of Theorem \ref{th:glue2}. }
\end{proposition}
\begin{proof} For $(p,x)\in \tilde Y_0 $
we showed in Proposition \ref{prop:connection} for $\mathcal X_p\in TY_0|_p$ and $V_x\in TF_x$ (using the same notation)
that
\begin{align}\label{ID}
&\left(\tilde \Pi_0+\tilde \Pi_1\right)_{(p,x)}({\mathcal X}_p,  V_x, {\mathcal  X}_p,  V_x)\\
&\ \ \ \ \ = \Pi_0(\mathcal X_p, \mathcal X_p) + \Pi_1(\mathcal X_p, \mathcal X_p) \nonumber\\
&\ \ \ \ \ \ \ \ \ \ + \underbrace{\frac{f^2(p) h(  V_x,  V_x)}{N-n}}_{\frac{\tilde g(\tilde V_{(p,x)}, \tilde V_{(p,x)})}{N-n}}\left(g_0( \nu_0, \nabla \Psi) + g_1( \nu_1, \nabla \Psi )\right)\nonumber\\
&{\color{black} \ \ \ \ \  \geq 
 - \tilde g(\tilde V_{(p,x)}, \tilde V_{(p,x)}) \left(g_0( \nu_0, \nabla \log \Phi^{\frac{1}{N-n}}) + g_1( \nu_1, \nabla \log \Phi^{\frac{1}{N-n}})\right)}\nonumber
\end{align}
where $\Psi := - \log f^{N-n}= -\log \Phi$ and $\tilde V$ is a vertical lift of $V$. {\color{black} We  used that $\Pi_0 + \Pi_1\geq 0$. }
\smallskip

 We set $\Pi_0+\Pi_1=:\Pi$ and $\tilde \Pi_0+\tilde \Pi_1=:\tilde \Pi$.  The trace of  $\tilde \Pi$ w.r.t. $\tilde g$  is 
\begin{align*} 
\mbox{tr}^{\tilde g} \tilde \Pi_{(p,x)}= \mbox{tr}^g\Pi_p +  \left(g_0( \nu_0|_p, \nabla \Psi|_p) + g_1( \nu_1|_p, \nabla \Psi |_p)\right)
\end{align*}
that is strictly positive on $\tilde Y_0\simeq \tilde Y_1$.
\smallskip

Recall the formula \eqref{id:ricci}. For the glued metric $\tilde g_\delta$ at some fixed point $(p,x)$ close to $Y_0\simeq Y_1$ it becomes

\begin{align}\label{id:another}
\mbox{Ric}_{\tilde{g}_\delta}(\tilde{\mathcal L})(\tilde \xi, \tilde \xi)&= \frac{1}{2} \frac{1}{\mu_n} \tilde \Pi(\tilde \xi,\tilde  \xi)+ \frac{1}{2} (\tilde \xi^n)^2 \sum_{l=1}^N \frac{1}{\mu_l} \tilde{\Pi}_{ll}.
\end{align}
Here, we choose coordinates $(x^1, \dots, x^n, \dots, x^N)$  such that $\tilde g_\delta$ is of diagonal form at $(p,x)$ and $(\mu_l)_{ l=1, \dots, N}$ are the diagonal entries. The index $n$  is  the coordinate direction $x^n$  that is normal to $Y_0\simeq Y_1$. $\tilde \xi= \xi + v$ is any vector in the tangent space at $(p,x)$.

Estimating the Ricci tensor of the metric $\tilde g_\delta$ corresponds to estimating $
\mbox{Ric}_{\tilde{g}_\delta}(\tilde{\mathcal L})(\tilde \xi, \tilde\xi)$.

We notice that the second term on the right hand side of \eqref{id:another} is
\begin{align*}
\frac{1}{2} (\tilde \xi^n)^2 \sum_{l=1}^N\frac{1}{\mu_n}\tilde{\Pi}_{ll} =\frac{1}{2}(\tilde \xi^n)^2\mbox{tr}^{\tilde g_\delta} \tilde \Pi
\end{align*}
Here 
$\mbox{tr}^{\tilde g_\delta}\tilde \Pi \rightarrow \mbox{tr}^{\tilde g}\tilde \Pi$  uniformly as $\delta\rightarrow 0$ . {\color{black} Since $\mbox{tr}^{\tilde g}\tilde \Pi> 0$ and since $\tilde Y_0\simeq \tilde Y_1$ is compact},   for $\delta>0$ sufficiently small, it follows that 
\begin{align*}{\color{black} \frac{1}{2}(\tilde \xi^n)^2\mbox{tr}^{\tilde g_\delta} \tilde \Pi> 0
 }
\end{align*}

{\color{black}
In the first term of the right hand side  of \eqref{id:another}   $\tilde \Pi(\tilde \xi, \tilde \xi)$ appears. Because of \eqref{ID} we can only estimate this with a correction that involves $\overline H$ and  the derivative  $f_\delta'$: If $f'_\delta \leq 0$, we can estimate $-2f_\delta' \tilde \Pi(\tilde \xi,\tilde  \xi)$ uniformly from below by $C (\min f_\delta') \bar H |v|_{\tilde g_\delta}^2$ where $C>0$ is another constant that does not depend on $\delta$. If $f'_\delta$ is positive, $f'_\delta$ can be estimated by $C\delta$. Moreover, since $Y$ is compact, there is a constant $A$ that uniformly estimates $\tilde \Pi$ (compare with Remark \ref{rem:2sided}).  In this case, we can estimate the term in question by $\epsilon(\delta)$ with $\epsilon(\delta)\downarrow 0$ for $\delta \downarrow 0$ (recall also the remark after \eqref{alsoimportant}).}

By step ${\bf (4)}$ in the proof of Theorem \ref{th:glue2} we get that
\begin{align*}\ric_{\tilde{g}^\delta}(\xi+v, \xi+v)&\geq
 (K-\epsilon(\delta))\left( |\xi|_{\tilde g^\delta}^2 + |v|_{\tilde g^\delta}^2\right) +{\color{black} (\min f'_\delta)} \overline H|v|_{\tilde g^\delta}^2
\end{align*}
 for $\xi+v\in T (\mathring C_0\cup_{\tilde{\mathcal I}}\mathring C_1)$.\end{proof}
By the same reasoning as in the proof of  Corollary \ref{cor:assum} we obtain. 
\begin{corollary}
For every $(p,x)\in \Int(\mathring C_0\cup_{\tilde{\mathcal I}}\mathring C_1)$ there exists $R>0$ such that $B^{\tilde g}_{2R}((p,x))\subset \Int (\mathring C_0 \cup_{\tilde{\mathcal I}}\mathring C_1)$ and $B^{\tilde g}_R((p,x))$ satisfies the condition $CD_{loc}(K+ (\min f'_\delta)\overline H,N)$. 
\end{corollary}
Let $(p,x)\in \Int(\mathring C_0\cup_{\tilde{\mathcal I}}\mathring C_1)$,  let $B_{2R}^{g^\delta}(p)\subset \Int(\mathring M_0\cup_{\mathcal I}\mathring M_1)$ and consider $\overline B_{R}^{g^\delta}\times r\mathbb S^{\cN-n}.$
\begin{proposition} The mm space
$$\left(\overline B_R^{g^\delta}(p)\times r\mathbb S^{\cN-n}, \tilde \m^\delta|_{\scriptstyle \overline B_R^{g^\delta}(p)\times r\mathbb S^{\cN-n}}\right)$$ converges in measured GH sense to $$\left(\overline B_R^{g^\delta}(p), \m^\delta|_{\overline B_R^{g^\delta}(p)}\right)$$ as $r\downarrow 0$ where
  $\tilde \m^\delta=\omega_{\cN-n}^{-1}{r^{-\lceil N \rceil+n}}\vol_{\tilde g^\delta}$ and 
$\m^\delta= \Phi^\delta \vol_{\hat g^\delta}$ with $$\frac{1}{\omega_{\cN-n}}\vol_{\tilde h^\delta(p, \cdot)}(\mathbb S^{\cN-n})= \Phi^\delta(p).$$
\end{proposition} \begin{proof}Since $\tilde g^\delta= g^\delta + r^2 \tilde h^\delta$ with $r>0$, it follows that $\overline B_{R}^{g^\delta}\times r\mathbb S^{\cN-n}$ converges in GH sense to $\overline B_{R}^{g^\delta}(p)$ as $r\downarrow 0$.  For $r>0$ an $\epsilon(r)$-GH-approximation (with $\epsilon(r)\downarrow 0$ for $r\downarrow 0$) is given by the projection of $\overline B_{R}^{g^\delta}(p)\times r \mathbb S^{N-n}$  onto $\overline B_{R}^{g^\delta}(p)$. 
 The volume form of $\tilde g^\delta$ is given by $$\vol_{\tilde g^\delta}= r^{\lceil N \rceil-n} \vol_{g^\delta} \vol_{h^\delta}.$$
The normalized volume form  $\tilde \m^\delta=\omega_{\cN-n}^{-1}{r^{-\lceil N \rceil+n}}\vol_{\tilde g^\delta}$ is therefore independent of $r$. The pushforward of $\tilde \m^\delta$ under the projection map is

$$
{ \frac{1}{\omega_{\cN-n}}\vol_{\tilde h^\delta(p, \cdot)}(\mathbb S^{\cN-n}) \vol_{g^\delta}}{=: \Phi^\delta(p)} \vol_{g^\delta}
$$
where $\Phi^\delta\in C^{\infty}(M)$ (since $\tilde h^\delta(p,x)$ is smooth w.r.t. $p$) and with $\Phi^\delta\geq 0$.  Since this measure does not depend on $r>0$, restricted to $B_{R}^{g^\delta}(p)$ it converges weakly to  $\m^\delta= \Phi^\delta \vol_{\hat g^\delta}$  restricted to $B_{R}^{g^\delta}(p)$ as $r\downarrow 0$.
\end{proof}
By the same reasoning as in the end of Remark \ref{rem:simple} one obtains the following corollary.
\begin{corollary}
The metric measure spaces $(C_0\cup_{\tilde{\mathcal I}}C_1, \tilde \m^\delta)$ locally converge in measured GH sense to $(M, d_{ g^\delta}, \m^\delta)$ as $r\downarrow 0$. 
\end{corollary}
\begin{theorem}  We fix $\delta>0$. With the same assumptions as in Proposition \ref{prop:gluedwarped2} it follows that
$(\mathring M_0\cup_{\mathcal I} \mathring M_1,  g^\delta, \Phi^\delta)$ has Bakry-Emery $\uN$-Ricci tensor bounded from below by $K-\epsilon(\delta)$.
\end{theorem}
\begin{proof}
The proof of Proposition \ref{prop:gluedwarped2}  yields that $\tilde g^\delta$ satisfies 
\begin{align}\label{ineq:H}\ric_{\tilde{g}^\delta}\!|_{(x,p)}(\xi+v, \xi+v)\geq (K-\epsilon(\delta))\left( |\xi|_{\tilde g^\delta}^2 + |v|_{\tilde g^\delta}^2\right) +{\color{black} (\min f'_\delta)}\overline H|v|_{\tilde g^\delta}^2.
\end{align}

Moreover $(\overline B_{4R}^{g^\delta}(p)\times r\mathbb S^{\cN-n}, \tilde \m^\delta|_{\overline B_{4R}^{g^\delta}(p)\times r\mathbb S^{\cN-n}})$ converges in measured GH sense to $(\overline B_{4R}^{g^\delta}(p), \m^\delta|_{\overline B_{4R}^{g^\delta}(p)})$ as $r\downarrow 0$ where $R>0$ such that $\overline B_{4R}^{g^\delta}(p)\subset \Int(\mathring M_0\cup_{\mathcal I}\mathring M_1)$. We will show the following claim.
\smallskip\\
{\it Claim:}
The condition $CD_{loc}(K-\epsilon(\delta),\cN)$ holds for $(B_R^{g^\delta}(p), \m^\delta|_{\overline B_R^{g^\delta}(p)})$. 
\smallskip\\
{\it Proof of the Claim:} We pick two measures $\mu_0, \mu_1\in \mathcal P(M, \m^\delta)$ supported in $ B_R^{g^\delta}(p)$. We consider the unique optimal coupling $\pi$ w.r.t. the cost function $d_{g^\delta}^2$.
We  assume  first that $\pi$ is concentrated on $\{(x,y)\in M^2: d_{g^\delta}(x,y)> c> 0\}=: \Xi$.   This will be removed at the end of the proof. 

By standard arguments  (for instance we refer  to \cite[Section 4]{gmsstability} and \cite[Section 7]{kettererlp}) there exist probability measures $\mu_0^r$ and $\mu_1^r$, supported in $\overline B_{R}^{g^\delta}\times r\mathbb S^{\cN-n}$ and absolutely continuous w.r.t. $\tilde m^\delta$ such that $\mu_i^r\rightarrow \mu_i$ weakly,  the $L^2$-Wasserstein distances satisfy $W_{\tilde g^\delta}(\mu_0^r, \mu_1^r)\rightarrow W_{g^\delta}(\mu_0, \mu_1)$ and the $W_{\tilde g^\delta}$ geodesic $(\mu_t^r)_{t\in [0,1]}$ between $\mu_0^r$ and $\mu_1^r$, supported in $\overline B_{4R}^{g^\delta}(p)\times r\mathbb S^{\cN-n}$ converges weakly to the $W_{g^\delta}$ geodesic between $\mu_0$ and $\mu_1$, supported in $\overline B_{2R}^{g^\delta}(p)$, as $r\downarrow 0$.  Moreover there exist optimal couplings $\pi^r$, $r>0$, that converge weakly to $\pi$ for $r\downarrow 0$.

\begin{jjj}
Note that  the weak convergence of these measure is understood in the following sense:
{\color{black} 
We consider an ambient metric space $(Y, d_Y)$ in which  $\overline B_{4R}^{g^\delta}(p)\times r\mathbb S^{\cN-n}$ and $\overline B_{4R}^{g^\delta}(p)$ equipped with the distances $d_{\tilde g^\delta}$ and $d_{g^\delta}$ embed preserving distances, and where  the measure Gromov-Hausdorff convergence  is realized. In $Y$ weak convergence of  measures is defined in duality with bounded continuous functions. }
\end{jjj}
 {\color{black}
Since $\pi$ is concentrated in  the open set $\{d_{Y}>c>0\}$, by weak convergence of $\pi^r$ to $\pi$ the measures $\pi^r$ are also supported in the open set $\{d_Y>c>0\}$.}

The Wasserstein geodesic between $\mu_0^r, \mu_1^r\in \overline B_{R}^{g^\delta}\times r\mathbb S^{\cN-n}$ is induced by a family of maps $t\in [0,1]\mapsto T^r_t(x)=:\gamma_x(t)=(\alpha_x(t), \beta_x(t))\in M\times \mathbb{S}^{\cN-n}$ such that $t\in [0,1]\mapsto \det DT_t^r(x)=:y_x(t)$ satisfies the Riccatti inequality
$$(\log y_x(t))''+ \frac{1}{N} ((\log y_x(t))')^2 +  \ric_{\tilde g^\delta}(\gamma'_x(t), \gamma_x'(t))\leq 0
$$
{\color{black} where $t\in [0,1]\mapsto \gamma_x(t)$ is a constant-speed, minimal geodesic with $\gamma_x(0)=x$, $\beta_x(t)$ is a pre-geodesic in $r\mathbb S^{\cN-n}$} and
$$\ric_{\tilde g^\delta}(\gamma_x'(t), \gamma_x'(t))\geq ( K-\epsilon(\delta)) |\gamma_x'(t)|_{\tilde g^\delta}^2 +{\color{black} { \frac{(\min f_\delta')\bar H|\beta_x'(t)|_{\tilde g^\delta}^2}{|\gamma_x'(t)|_{\tilde g^\delta}^2}}} |\gamma_x'(t)|_{\tilde g^\delta}^2 .$$
This is derived by standard Jacobi field calculus for the family $t\in [0,1]\mapsto DT_t(x)$. For details we refer to \cite[Proof of Theorem 1.7]{stugeo2} and to \cite{cms}.  Similar calculation were performed in \cite{kettererlp}.

Since $\pi^r$ is concentrated on $\{ d_Y>c>0\}$, it holds $|\gamma'(t)|_{\tilde g^\delta}=L^{\tilde g^\delta}(\gamma)\geq c>0$.  Moreover $|\beta'(t)|_{\tilde g^\delta} \leq \tilde C r$ for a constant $\tilde C>0$ that is independent of $r$.  Hence, for a constant $C>0$ that also does not depend on $r$,
 $${ \frac{({\color{black} \min f_\delta'})\bar H|\beta_x'(t)|_{\tilde g^\delta}^2}{|\gamma_x'(t)|_{\tilde g^\delta}^2}} \geq -C r^2.$$

By a standard procedure for which we again refer to \cite{stugeo2, cms} it follows that the $W_{\tilde g^\delta}$ geodesic $(\mu_t^r)_{t\in [0,1]}$ that is induced by $T_t^r$ between $\mu_0^r$ and $\mu_1^r$ satisfies \eqref{ineq:cd} in Definition \ref{def:cd} with $K$ replaced with $K-\epsilon(\delta)- Cr^2$. This inequality is stable under the convergence of $\mu_t^r$ to $\mu_t$. It follows that the Wasserstein geodesic $\mu_t$ between $\mu_0$ and $\mu_1$ satisfies the inequality with $K-\epsilon(\delta)$. 

{\color{black}
Finally, we remove the assumption on the support of $\pi$. Let $\mu_0, \mu_1\in \mathcal P(M, \m^\delta)$ concentrated in $B_R^{g^\delta}(p)$ and let $\pi$ be  the associated optimal coupling. We note that $$\pi= \pi|_{\{d_{g^\delta}>0\}} +  \pi|_{\{d_{g^\delta}\equiv 0\}}= A \pi' + B\pi''$$
where $A=\pi(\{\de_{g^\delta}>0\})$, $B=\pi(\{\de_{g^\delta}\equiv 0\})$, and $\pi'$ and $\pi''$ are optimal couplings. We assume that $A\neq 0$. We obtain two disjoint Wasserstein geodesics  $\mu'_t$ and $\mu''_t$ that connect the marginal distributions of $\pi'$ and $\pi''$, respectively.  Moreover  $S_N(\mu_t| \m^\delta)= A^{1-\frac{1}{N} }S(\mu_t'|\m^\delta) + B^{1-\frac{1}{N}}S(\mu_t''|\m^\delta)$ and the displacement inequality along $\mu_t''$ is an equality.  It is therefore  sufficient to check  the displacement inequality for the optimal coupling $\pi'$ and its marginal distributions $\mu_0'$ and $\mu_1'$. 
We define $$\pi^k:=\pi'(\Xi^k)^{-1} {\pi'}|_{\Xi^k}, \ \ \Xi^k=\{d_{g^\delta}(x,y)>1/k>0\}.$$ 
The probability measure $\pi^k$ is  an optimal coupling between its marginal distributions $(P_i)_\sharp \pi^k= \mu_i^k$, $i=0,1$, that are  $\m^\delta$-absolutely continuous.  We consider the displacement inequality for $\mu_t^k$, the associated Wasserstein geodesic between $\mu_0^k$ and $\mu_1^k$. From the previous steps we have that $\mu^k_t$ and $\pi^k$ satisfy the displacement inequality with  the curvature parameter $K-\epsilon(\delta)=:K_\epsilon$, i.e. 
\begin{align*}
S_N(\mu_t^k|\m^\delta)\leq -\int \left[\tau_{K_\epsilon,N}^{(1-t)}(\theta)(\rho^k_0(x))^{-\frac{1}{N}}+\tau_{K_\epsilon,N}^{(t)}(\theta)(\rho^k_1(y))^{-\frac{1}{N}}\right]d\pi^k(x,y)
\end{align*}
where $\mu_i=\rho_id\m$, $i=0,1$, and $\theta= d_{g^\delta}(x,y)$. 

We observe that the right hand side of the last inenquality can be estimated from above by 
$$ -\int \left[\tau_{K_\epsilon,N}^{(1-t)}(\theta)(\rho'_0(x))^{-\frac{1}{N}}+\tau_{K_\epsilon,N}^{(t)}(\theta)(\rho'_1(y))^{-\frac{1}{N}}\right]d\pi'(x,y) + C \int_{\Xi^k} d\pi'$$
where the second term vanishes for $k\rightarrow \infty$. This estimate is obtained along the same lines as in step (iv) of the proof of Theorem 3.1 in \cite{stugeo2}.

Since $\overline{B_R^{g^\delta}(p)}$ is compact and since $\pi'(\{d_{g^\delta}>1/k\})\rightarrow 1$, after extracting a subsequence, we find that 
%
that  $\mu^k_t$ converges to a Wasserstein geodesic $\tilde \mu_t$ between $\mu_0'$ and $\mu_1'$.  Since $\mu_0'$ and $\mu_1'$ are $\vol_{g^\delta}$-absolutely continuous and since the underlying space is a smooth weighted Riemannian manifold, the Wasserstein geodesic between them is unique. Hence $\mu_t'= \tilde \mu_t$.  Moreover \begin{center}$\liminf_{k\rightarrow \infty} S_N(\mu^k_t|\m^\delta) \geq S_N(\mu_t'|\m^\delta)$. \end{center}
Hence, the desired inequality follows $(\mu_t')_{t\in [0, 1]}$ and $\pi'$.

This finishes the proof of the claim. }
 \hfill $\triangle$
\smallskip

The statement of the theorem follows: Since $$\left(B_R^{g^\delta}(p), \m^\delta|_{\overline B_R^{g^\delta}(p)}\right)$$ is  a smooth weighted Riemannian manifold  ($g^\delta$ is the smooth approximation of the glued metric and $\m^\delta= \Phi^\delta \vol_{g^\delta}$ for $\Phi^\delta$ smooth), it follows that the corresponding Bakry-Emery $\cN$-tensor is pointwise bounded from below $K-\epsilon(\delta)$. This is the assertion of the theorem.
%
\end{proof}
\begin{proof}[Proof of Theorem \ref{main3}] We fix a sequence of $\delta_n=\delta$. {\color{black} For each $n$ we can deform  the weight $f$ slightly to $\tilde f$ such that Assumption \ref{ass3} holds in a strict sense while preserving the assumption on the Bakry Emery Ricci tensor up to an error $\epsilon_n$ such that $\epsilon_n\rightarrow 0$ (compare with Proposition 1.2.11 in \cite{burdickthesis}).}

We can apply the previous theorem. {\color{black} By a diagonal argument and 
since ${\bf G}_{\delta_n}\rightarrow {\bf I}$ uniformily on $M$ w.r.t. $g_0$, we can find a subsequence such that $g^{\delta_n}$ as well as $\tilde \Phi^{\delta_n}$ converge uniformily to $g$ and $\Phi$, respectively,  as $\delta\downarrow 0$. } Hence, the family of weighted Riemannian manifolds $(M,  g^{\delta_n}, \tilde \Phi^{\delta_n})$ satisfies the desired properties.
%
\end{proof}
\begin{proof}[Proof of Theorem \ref{main2}]
We   bootstrap the whole argument. We  apply the warped product construction of Proposition \ref{prop:warped} with  $(M, \hat g^\delta)$ instead of $(M,g)$ and $\tilde \Phi^\delta$ instead of $\Phi$. Note that the boundary conditions of $(M, g^\delta, \Phi^\delta)$ on $\partial M$ are the same as for $(M_i, g_i, \Phi_i)$ on $\partial M_i \backslash Y_i$. In particular, we can  apply the compactness theorem of \cite{wong}. This yields the desired sequence of Riemannian manifolds. 
\end{proof}
\section{Proof of Corollary 1.3}\label{sec: Second application}
%
%
%
%
%
%
%
%
%
%
%
%
%
%
%
%
%
We assume the situation of Theorem \ref{main2} together with $\Pi_i\geq 0$ on $M_i\backslash Y_i$.  
In this case we already know that the weighted glued space $(M_0\cup_{\mathcal I}M_1, \Phi)$ satisfies $CD(K,\lceil N \rceil)$ because of Theorem \ref{main2} and we want to replace $\lceil N \rceil$ with $N$. 
We follow  the strategy that was applied in \cite[Section 3]{kakest}. The main difference  is that $M_0$ and $M_1$ equipped with $g_0$ and $g_1$ are not Alexandrov spaces in general since the second fundamental forms of the  boundaries are not positive semi-definite.  
\subsection{An application of $1D$ localisation}
The  results of this section are motivated by the next example for the behavior of geodesics in a glued space. 

\begin{example}
Consider the smooth funktion $f: \R\rightarrow \R$ given by 
$$
f(x)= \begin{cases}
e^{-\frac{1}{x}}\sin\left({e^{\frac{1}{x}}}\right) & \mbox{ if } x>0, \\
0& \mbox{ otherwise}.
\end{cases}
$$
Then we define two manifolds with boundary through $M_0= \{(x,y)\in \R^2: y\geq f(x)\}$ and $M_1=\{(x,y)\in \R^2: y\leq f(x)\}$.  Since $f$ is smooth also the boundaries of $M_0$ and $M_1$ are smooth. Moreover, the sum of the second fundamental forms trivially satisfies the condition in Assumption \ref{ass:1}. The glued space $M$ is $\R^2$.  One geodesic in $M$ is given by the $x$-axis in $\R^2$ and hence it oscillates between $M_0$ and $M_1$ when approaching $0$ from the right hand side. The example ilustrates that under the Assumption \ref{ass:1} such oscillating behavior of geodesics in general cannot be avoided. 
However  we will show in the following that geodesics with this behaviour are  a "set of measure $0$".
\end{example}
We note that by the Assumption \ref{ass:1_0}, $\Pi_i\geq 0$ on $M\backslash Y_i$, Theorem \ref{th:basicglue} and Theorem \ref{th:petrunincd} we know that the metric glued space is an Alexandrov space with curvature bounded from below by $\underline k$.

Let $u: M\rightarrow \R$ be a $1$-Lipschitz function, let $(\m_{\gamma})_{\gamma\in Q}$ be the induced disintegration {of $\vol_{g}$} together with the quotient measure $\mathfrak q$.
We pick  $\hat Q$ of full $\mathfrak q$ measure in $Q$ such that $R_u(x)=\overline{X_\gamma}$ for all $x\in X_\gamma$ and for all $\gamma\in \hat Q$. By abuse of notation we write $\hat Q=Q$ and $\T_u=\mathfrak Q^{-1}(\hat Q)$ (see the comments after Lemma \ref{somelemma}).

We say that  a unit speed geodesic $\gamma:[a,b]\rightarrow X$ is tangent to $Y$ if
 there exists $t_0\in [a,b]$ such that $\gamma(t_0)\in Y$ and $\dot{\gamma}(t_0)\in C_pM$ is tangent to the boundary.
\begin{fact}\label{lem: infinite implies tangent}
If $\gamma\notin Q^{\dagger}:= \left\{\gamma\in Q: \# \gamma^{-1}(Y)<\infty\right\}$, then $\gamma$ is tangent to $Y$.
\end{fact}
\begin{proof}
There exists a sequence  $t_k\in [a_\gamma, b_\gamma]$ such that $\gamma(t_k)\in Y$ are pairwise distinct.  After we choose a subsequence we have $t_k\rightarrow t_0$, $t_k>t_0$ and $\gamma(t_i){\rightarrow}\gamma'(t_0)=v\in C_pX$ in GH sense. 
%
We also have that $d_{\partial M_i}(\gamma(t_k))= 0$, $i=0,1$, since $\gamma(t_k)\in Y$. Then $d(d_{\partial M_i})|_p(v)= g|_{\gamma(t_0)}(\nabla d_{\partial M_i}|_{\gamma(t_0)}, v)= 0$. Hence $v$ is tangent to the boundary and consequently $\gamma$ is tangent to $Y$. 
\end{proof}

\begin{remark}
For $U\subset M$ open we write
\begin{align*}
{\vol_g({\T_u\cap U})}=\int_Q\m_{\gamma}(U) d\mathfrak q(\gamma) = \int_{\mathfrak G^{-1}(U)} h_\gamma(r) dr\otimes d\mathfrak q(\gamma)
\end{align*}
where $\mathfrak G:\mathcal V\subset \R\times Q\rightarrow \T^b_u$ is the ray map defined in Subsection \ref{subsec:1Dlocalisation}. 
We also note that $(r,\gamma)\in \mathcal V \mapsto h_\gamma(r)$ and $\gamma\in Q\mapsto a_{\gamma}, b_{\gamma}$ are measurable. 
\end{remark}

\begin{remark}
Consider the map $\Phi_t: \R\times Q\rightarrow \R\times Q$, $\Phi_t(r,q)=(tr,q)$ for $t>0$. Then, it is clear that $\Phi_t(\mathcal V)=\mathcal V_t$ is a measurable subset of $\mathcal V$ for $ t\in (0,1]$. Moreover $\mathfrak G(\mathcal V_t)=\T^b_{u,t}$ is  measurable  and a subset of $\T^b_u$ such that $X_\gamma\cap \T^b_{u,t}= t X_\gamma \subset X_\gamma$. If $t\in (0,1)$, then $\vol_g(\T^b_u\backslash \T^b_{u,t}) >0$.
Again by Fubini's theorem  $U\cap X_\gamma \cap \T_{u,t}^b= U\cap t X_\gamma$ is measurable in $X_\gamma$ for $\mathfrak q$-a.e. $\gamma\in Q$ and the map $$L_{U,t}:\gamma\in Q\mapsto  \mbox{L}(\gamma|_{(t a_\gamma, tb_\gamma)\cap \gamma^{-1}(U)})=\int 1_{U\cap tX_\gamma} d\mathcal L^1$$ is measurable. The set $(ta_\gamma, tb_\gamma)\cap \gamma^{-1}(U)$ might not be an interval.

\end{remark}
\begin{definition}
Consider $U_\eps=B_{\epsilon}(Y)$ for $\epsilon>0$. For $s\in \N$ and $t\in (0,1]$ we define
\begin{align*}
C_{\epsilon, s, t}=\left\{\gamma\in Q:  \mbox{L}(\gamma|_{\gamma^{-1}(U_\eps)\cap (ta_{\gamma},tb_\gamma)})>\epsilon s\right\}
\end{align*} 
as well as 
\begin{align*}
C_{s,t}=\bigcup_{\epsilon>0}\bigcap_{\epsilon'{{\geq}}\epsilon} C_{\epsilon',s,t}=\{\gamma\in Q: \liminf_{\epsilon\rightarrow 0}\mbox{L}(\gamma|_{\gamma^{-1}(U_\epsilon)\cap (ta_\gamma, tb_\gamma)})/\epsilon \geq s\}
\end{align*}
and 
\begin{align*}
 C_t=\bigcap_{s\in \N}C_{s,t} =\{\gamma\in Q: \lim_{\epsilon\rightarrow 0}\mbox{L}(\gamma|_{\gamma^{-1}(U_\epsilon)\cap (ta_\gamma, tb_\gamma)})/\epsilon=\infty\}. 
\end{align*}
\end{definition}
\begin{lemma}\label{lemma: tan-t->Ct}
Let  $t\in (0,1]$, $\gamma\in Q$. If $\gamma|_{[ta_\gamma, tb_\gamma]}$ is tangent to $Y$, then $\gamma\in C_t$.
\end{lemma}
\begin{proof}
For the proof we ignore $t\in (0,1]$ and consider $\gamma|_{[a_\gamma,b_\gamma]}$.
Let $\gamma\in Q$ be tangent to $Y$. Assume $\gamma\notin C$.
Then there is a sequence $(\epsilon_i)_{i\in \N}\downarrow 0$ such that $$\lim_{i\rightarrow \infty} {\L(\gamma|_{\gamma^{-1}(B_{\epsilon_i}(Y))\cap (a_\gamma,b_\gamma)})}/{\epsilon_i}=:\lambda\in [0,\infty)$$ and hence   $\lim_{i\rightarrow \infty} {\L(\gamma|_{\gamma^{-1}(B_{\epsilon_i}(Y))\cap (a_\gamma,b_\gamma)})}=:0$. 
By assumption there exists $t_0\in [a_\gamma, b_\gamma]$ such that $\gamma(t_0)\in Y$. Hence $t_0\in \gamma^{-1}(B_{\epsilon_i}(Y))\cap [a_\gamma,b_\gamma]$. 
There exists a maximal interval $I^{\epsilon_i}$ contained in $\gamma^{-1}(B_{\epsilon_i}(Y))\cap [a_\gamma,b_\gamma]$  such that $t_0\in I^{\epsilon_i}$. 
Then
$$\infty>\lambda=\lim_{i\rightarrow \infty} \frac{\L(\gamma|_{\gamma^{-1}(B_{\epsilon_i}(Y))\cap (a_\gamma,b_\gamma)})}{\epsilon_i}\geq \liminf_{i\rightarrow \infty }\frac{\L(\gamma|_{I^{\epsilon_i}})}{\epsilon_i}\geq 0$$ 
and  $\L(\gamma|_{I^{\epsilon_i}})=:L_i\rightarrow 0$. We set $\gamma_i=\gamma|_{I^{\epsilon_i}}$. Since $I^{\epsilon_i}$ is maximal such that $\mbox{Im}(\gamma_i)\subset B_{\epsilon_i}(Y)$, we have 
$\sup_{t\in I^{\epsilon_i}}\inf_{y\in Y}d(y,\gamma(t))=\sup_{t\in I^{\epsilon_i}}\de_Y(\gamma(t))\geq \epsilon_i$ where $\de_Y=\inf_{y\in Y} \de(y, \cdot)$ is the distance function to $Y$.  Otherwise we would have $(a_\gamma, b_\gamma)\subset \gamma^{-1}(B_{\epsilon_i}(Y))$ for all $i\in \mathbb N$, and hence $\gamma\in C$. Since $\de_Y$ is continuous and $\overline{I^{\epsilon_i}}$ is compact, there exist $t_i\in \overline{I^{\epsilon_i}}$ such that $\de_Z(\gamma(t_i))= \sup_{t\in I^{\epsilon_i}}\de_Y(\gamma(t))= \epsilon_i$.  Moreover, $L_i\geq 2\epsilon_i$ since $I^{\epsilon_i}$ is maximal. Hence $\frac{L_i}{\epsilon}\geq 2$.
In the rescaled space $(Z, \frac{1}{L_i} d_Z)$ the geodesic $\gamma_i$ is of length $1$ and $$\frac{1}{2}\geq \frac{1}{L_i}\de_Y(\gamma(t_i))=\frac{\epsilon_i}{L_i}\geq \lambda/2$$
for $i\in \N$ sufficiently large.
$(M, \frac{1}{L_i}d_g, \gamma(t_0))$ converges in pointed GH sense to $T_{\gamma(t_0)}M$ and a subsequence of $\gamma(t_i)$ converges to   $\eta \dot\gamma(t_0)\in T_{\gamma(t_0)}M$ in GH sense with $\eta\in [0,2]$. Moreover $\frac{1}{L_i} \de_Y(\gamma(t_i))\rightarrow d(\de_{Y})_{\gamma(t_0)}(\dot\gamma(t_0))\geq \lambda/2$.
This is a contradiction, since either we have $\eta=0$, or $\eta\neq 0$ but $\dot\gamma(t_0)$ was assumed to be tangent to $Y$.
Hence, for any sequence $\epsilon_i\rightarrow 0$ it follow that $\frac{L_i}{\epsilon_i}\rightarrow \infty$ and therefore $\gamma \in C_t$. %
\end{proof}
\begin{corollary}\label{cor: tan->Ct}
Let $\gamma\in Q$. If {$\gamma|_{(a_\gamma,b_\gamma)}$} is tangent to $Y$, then $\gamma\in {C= \bigcup\limits_{t\in (0,1)}}C_t$. 
\end{corollary}
\begin{lemma}\label{lemma:tangentgeodesics}
 $\mathfrak q(Q\cap \bigcup_{t\in (0,1)} C_t)=0$.
\end{lemma}

\begin{proof} It is  enough to show  that $\mathfrak q(Q\cap C_t)=0$ for any $t\in (0,1)$. Therefore we fix $t\in (0,1)$ in the following.
We recall that {$a_\gamma<0<b_\gamma$},  $\gamma\in Q\mapsto a_\gamma, b_\gamma$ are measurable and $Q=\bigcup_{l\in \N}\{l\geq |b_\gamma|, |a_\gamma| \geq \frac{1}{l}\}$.
It is  enough to prove the lemma for $Q^l=\{l\geq |a_\gamma|, |b_\gamma|\geq \frac{1}{l}\}$  for arbitrary $l\in \N$. Therefore
we fix $l\in \mathbb N$ and replace $Q$ with $Q^l$. We will drop the superscript $l$ for the rest of the proof.
By rescaling the whole space with $4l$ we can assume that {$4\le |a_\gamma|, |b_\gamma| \le 4l^2$ }for each $\gamma\in Q$.

Let $C_{\epsilon, s, t}$ be defined as  before for $\epsilon\in (0,\epsilon_0)$ and $s\in \N$.
We pick $\gamma\in C_{\epsilon,s, t}$ and consider $\gamma^{-1}(B_{\epsilon}(Y))\cap (ta_\gamma, tb_\gamma)=: I_{\gamma, \epsilon}$. We set $\L(\gamma|_{I_{\gamma, \epsilon}})=:L^{\epsilon}$.

We observe that $4l^2\geq (1-t)|a_\gamma|\geq (1-t)4\mbox{ and }  4l^2\geq (1-t)|b_\gamma| \geq (1-t)4.$
We pick $r\in I_{\gamma, \epsilon}$ and $\tau \in (a_\gamma, ta_\gamma)\cup (tb_\gamma, b_\gamma)$. Theorem \ref{th:1dlocalisation} implies that $([a_\gamma, b_\gamma], h_\gamma dr)$ satisfies the condition $CD(k(n-1),n)$. Then, the following estimate holds (c.f. \cite[Inequality (4.1)]{cavmon})
\begin{align*}
h_\gamma(r)&\geq \frac{\sin^{n-1}_{k}((r-a_\gamma)\wedge (b_\gamma-r))}{\sin^{k-1}_{k}((\tau-a_\gamma)\wedge (b_\gamma-\tau))}{h_\gamma(\tau)} \\ &\geq \frac{\sin^{n-1}_{k}((1-t)4)}{\sin^{n-1}_{k}4l^2} h_\gamma(\tau)\\
&= C(k,n,t,l) h_\gamma(\tau)
\end{align*}
for a universal constant $C(k,n,t,l)$.
We take the mean value w.r.t. $\mathcal L^1$ on both sides and obtain
\begin{align*}
\frac{1}{L^{\epsilon}}\int_{I_{\gamma, \epsilon}} h_\gamma d\mathcal L^1 \geq C(k,n,t,l) \frac{1}{4l^2}\int_{(a_\gamma, ta_\gamma)\cup (tb_\gamma, b_\gamma)} h_\gamma d\mathcal L^1.
\end{align*}
Hence, after integrating w.r.t. $\mathfrak q$ on $C_{\epsilon,s,t}$ and taking into account $\frac{1}{\epsilon s}\geq \frac{1}{L^{\epsilon}}$ by definition of $C_{\epsilon,s,t}$, it follows
\begin{align*}
\frac{1}{\epsilon s}\vol_g(B_{\epsilon}(Y))&\geq 
\frac{1}{s\epsilon}\int_{C_{\epsilon,s,t}}\m_{\gamma}(B_{\epsilon}(Y))d\mathfrak q(\gamma)\\
&\geq\frac{1}{L^{\epsilon}}\int_{C_{\epsilon,s,t}} \int_{I_{\gamma, \epsilon}} h_\gamma d\mathcal L^1 d\mathfrak q(\gamma) \\
&\geq \hat C \int_{C_{\epsilon,s,t}}\int_{(a_\gamma, ta_\gamma)\cup (tb_\gamma, b_\gamma)} h_\gamma d\mathcal L^1 d\mathfrak q(\gamma)\\
&\geq \hat C \int_{C_{\epsilon,s,t}} \m_\gamma(\T^b_u\backslash \T^b_{u,t})d\mathfrak q(\gamma)
\end{align*}
where $\hat C=\frac{1}{2l} C(k,n,t,l)$.

Since $(M, g, \vol_g)$ satisfies $CD(k(n-1),n)$ and since $Y$ is a smooth embedded, compact submanifold of $M_0$ and of $M_1$, it follows from the Heintze-Karcher inequality for $CD$ spaces (see \cite{kettererhk}) that $\vol_g(B_{\epsilon}(Y))\leq \epsilon M$ for a constant $M>0$ that only depends on $k$, $n$ and a lower bound of the second fundamental forms of $Y$ in $M_0$ and in $M_1$ respectively. 
Hence
\begin{align*}
\frac{M}{s} \geq C(K,N,k,t) \int_{C_{\epsilon,s,t}} \m_\gamma(\T^b_u\backslash \T^b_{u,t})d\mathfrak q(\gamma).
\end{align*}
If we take the limit for $\epsilon\rightarrow 0$, we obtain

\begin{align*}
\frac{M}{s} \geq C(K,N,k,t) \int_{C_{s,t}} \m_\gamma(\T^b_u\backslash \T^b_{u,t})d\mathfrak q(\gamma).
\end{align*}
Finally, for $s\rightarrow \infty$ it follows
\begin{align*}
0 = \int_{C_{t}} \m_\gamma(\T^b_u\backslash \T^b_{u,t})d\mathfrak q(\gamma).
\end{align*}
But by construction of $\T^b_{u,t}$ we know that $\m_\gamma(\T^b_u\backslash \T^b_{u,t})$ is positive for every $\gamma\in Q$ if $t\in (0,1)$. Therefore, it follows $\mathfrak q(C_t)=0$.
\end{proof}
Combining the above lemma with Corollary~\ref{cor: tan->Ct} gives
\begin{corollary}\label{cor: tan-measure 0} $\mathfrak q(\gamma\in Q:  \gamma|_{(a_\gamma,b_\gamma)}$ tangent to $Y)=0$.
\end{corollary}
\begin{remark} It is not claimed that the set of geodesics in $Q$  which are tangent to $Y$ at one of the endpoints has measure zero. 
\end{remark}

{\begin{corollary}\label{cor:important}
Let $x_1\in M_1\backslash Y$. Then, {for} $\vol_{g}$-a.e. point $x_0\in M_0$ there exists a unique geodesic that connects $x_0$ and $x_1$ and intersects with $Y$ only finitely many times. 
\end{corollary}
\begin{proof}Consider the $1$-Lipschitz function $u=\de(\cdot, x_1)$  and the corresponding $1D$ localisation with the quotient space $Q$ of geodesics, the quotient map $\mathfrak Q$  and the quotient measure $\mathfrak q$. 
Corollary \ref{cor: tan-measure 0} yields that there exists $\hat Q\subset Q$ such that $\mathfrak q(Q\backslash \hat Q)=0$ and for every geodesic $\gamma\in \hat Q$ we know that it intersects with $Y$ only finitely many times.
Also $\mathfrak Q^{-1}(\hat Q)\cap X_0$ has full $\vol_g$-measure in $M_0$. 
\end{proof} }
%
%
%
%
%

\subsection{} Here we finish the proof of Corollary \ref{cor1}.

{\textbf 1. } Let $(M_i, \Phi_i)$, $i=0,1$, be weighted Riemannian manifolds with $\ric^{\Phi_i, N}_{M_i}\geq K$ satisfying Assumption \ref{ass:1} and Assumption \ref{ass:3}.

By Theorem \ref{main3} and stability of the curvature-dimension condition it follows that the metric glued space $M_0\cup_{\mathcal I}M_1$ with the reference measure $\Phi \vol_g$ 
satisfies the condition $CD(K,\cN)$.

Hence, any $1$-Lipschitz function $u: M_0\cup_{\mathcal I} M_1\rightarrow \R$ induces a disintegration $\{\m_\gamma\}_{\gamma\in Q}$ that is strongly consistent with $R^b_u$, and for $\mathfrak q$-a.e. $\gamma\in Q$ the metric measure space $(\overline{X}_\gamma,\m_\gamma)$ satisfies the condition $CD(K, \cN)$.
It follows that $\m_{\gamma}=h_\gamma \mathcal H^1|_{X_\gamma}$ and $h_\gamma:[a_\gamma, b_\gamma]\rightarrow \R$ satisfy
\begin{align}\label{ano}
\frac{d^2}{dr^2} h_\gamma^{\frac{1}{\cN-1}}+ \frac{K}{\lceil N\rceil-1} h_\gamma^{\frac{1}{\cN-1}}\leq 0 \ \mbox{ on } (a_\gamma, b_\gamma) \ \mbox{ for }\mathfrak q\mbox{-a.e.}\gamma\in Q.
\end{align}
By Lemma \ref{importantlemma} it follows
\begin{align*}
\frac{d^-}{dr} h_{\gamma}^{\frac{1}{\cN-1}}\geq \frac{d^+}{dr} h_\gamma^{\frac{1}{\cN-1}} \ \text{  on }(a_\gamma, b_\gamma)
\end{align*}
and hence 
\begin{align*}
\frac{d^-}{dr} h_{\gamma}^{\frac{1}{N-1}}\geq \frac{d^+}{dr} h_\gamma^{\frac{1}{N-1}} \ \text{  on }(a_\gamma, b_\gamma).
\end{align*}
{\bf 2.}  Fix $0<t<1$. Define the set $C_t$ as in Section~\ref{sec: Second application}. 
Recall that all points in $M$ are regular points in the sense of an Alexandrov space. 

Let $Q_t= Q\backslash C_t$. By Lemma~\ref{lem: infinite implies tangent} and Lemma~\ref{lemma: tan-t->Ct} we know that for any $\gamma\in Q_t$ it holds that $\gamma|_{(ta_\gamma,tb_\gamma)}$ intersects $Y$ in finitely many points. Further by Lemma ~\ref{lemma:tangentgeodesics} we know that $Q_t$ has full measure in $Q$.
Since the Bakry-Emery $N$-tensor of $(M_i, \Phi_i)$, $i=0,1$, is bounded from below by $K$ we can follow precisely the final arguments in \cite{kakest} to obtain that
for $\mathfrak q$-almost every $\gamma\in Q$ the inequality \eqref{ano} holds with $N$ replaced by $\lceil N \rceil$ 
for any interval $I\subset (a_\gamma, b_\gamma)$ as long as $\gamma|_I$ is fully contained in $M_i$ for some $i=0,1$.

From Lemma \ref{importantlemma} 
 it follows that  \eqref{ano} with $\lceil N\rceil$ replaced by $N$ holds 
on $(ta_\gamma, tb_\gamma)$ for any $\gamma\in Q_t$.
Since this holds for arbitrary $0<t<1$, we get that for $\mathfrak q$-almost all $\gamma$ in $Q$ it holds that $([a_\gamma,b_\gamma], m_\gamma)$ satisfies $CD(K,N)$.
Since this holds for an arbitrary 1-Lipschitz function $u$ we obtain that $(M_0\cup_{\mathcal I} M_1, \Phi)$ satisfies $CD^1_{lip}(K,N)$.
Hence Theorem  \ref{thm:cavmil} yields the condition $CD(K,N)$ for $(M_0\cup_{\mathcal I} M_1, \Phi \vol_g)$. 

\section{Proof of Theorem 1.4}\label{sec:6}
\begin{th:re}
Assume the metric glued space $M_0\cup_{\mathcal I} M_1$ equipped with $\m= \Phi \vol_g$ satisfies a curvature-dimension condition $CD(K,N)$ for $K\in \R$ and $N\in [1, \infty)$. Then it follows
\smallskip
\begin{itemize}
\item[(1)] $\Pi_i \geq 0$ on $\partial M_i \backslash Y_i,$
\medskip 
\item[(2)] $\Pi_1 + \Pi_2 =: \Pi \geq 0$ on $Y_0\simeq Y_1$,  
\medskip
\item[(3)] $ \tr \Pi- \langle N_0, \nabla \log \Phi_0\rangle - \langle N_1, \nabla \log \Phi_1\rangle\geq  0$ on $Y_0\simeq Y_1$.
\end{itemize}
\end{th:re}
\begin{proof}
{\bf 1.} We first prove (3). Let 
\begin{center}
$H^{\Phi_i}
:= \tr \Pi_i -\langle N_i, \nabla \log \Phi_i\rangle 
$ on $Y_0\simeq Y_1$, $i=0,1$.
\end{center}
Let $d_{Y_0}$ and $d_{Y_0}$ be the distance functions of $Y_0\simeq Y_1=:Y$ in $M_0$ and in $M_1$ respectively. Since $Y$ is compact, for $\delta>0$ small enough restricted to $B_\delta(Y)$ the distance functions are smooth and $\nabla d_{Y_i}|_{Y}= N_i$ is the inward unit normal vector field of $Y_i$ in $(M_i, g_i)$, $i=0,1$.  The signed distance function $d_Y$ in the metric glued space $M_0\cup_{\mathcal I} M_1$ is  given by
$$d_Y= d_{Y_0}- d_{Y_1}$$
where $d_Y|_{M_0}\geq 0$ and $d_Y|_{M_1}\leq 0$, and $d_Y(x)=0$ if and only if $x\in Y$. Since the glued space is a geodesic metric space, $d_Y$ is a $1$-Lipschitz function and the induced disintegration $\m_{\gamma}= h_\gamma(t) dt$, $\gamma\in Q$, of $\m$ according to Subsection \ref{subsec:1Dlocalisation} can be constructed explicitly as follows:
\smallskip

Let $\gamma_x:(-\delta, \delta)\rightarrow M_0\cup_{\mathcal I} M_1$ be the geodesic passing through $\gamma_x(0)=x\in Y$ such that $\frac{d}{dt^+}\Big|_0 \gamma_x(t)= N_0$ and $\frac{d}{dt^-}\Big|_0\gamma_x(t)=N_1$. 

The map $T: Y\times (-\delta, \delta)\rightarrow M$ defined via $T(x,t)=\gamma_x(t)$ is smooth, a diffeomorphism  for $\delta>0$ sufficiently small and integrals on $B_\delta(Y)$ can be computed effectively via 
$$\int g d\m= \int g\circ T(x,t) \det (DT_{(x,t)}|_{T_xY}) \Phi \circ T(x,r) dr d\vol_Y(x).$$
This is a disintegration of $\m$  that is strongly consistent w.r.t. $\{\mbox{Im}(\gamma_x)\}_{x\in Y}$. It follows that 
$$\m_x =h_x(t) dt= \frac{1}{\lambda_x} \det (DT_{(x,t)}|_{T_xY}) \Phi \circ T(x,r) dr$$
and the quotient space $Q$ is given by $Y$.  $\lambda_x$ is  a normalisation constant such that $\m_x$ are probability measures. 
One can compute easily (as in \cite{kettererhk}) that 
$$ \frac{d}{dt^+}\Big|_{t=0} \log h_x(t) = H^{\Phi_0}(x)\ \mbox{ and } \ \frac{d}{dt^-} \Big|_{t=0} \log h_x(t)= H^{\Phi_1}(x).$$
The density  $h_x(t)$ is semi-convex by the Theorem \ref{thm:cavmil} about the characterizaton of the curvature-dimension condition and hence $H^{\Phi_0}+ H^{\Phi_1}\geq 0$ along $Y_0\simeq Y_1$. 
\smallskip\\
{\bf 2.} In the remaining  steps we prove  (2).  (1) follows already from \cite{han19}. 
\smallskip

We recall that   $g_0|_Y \simeq g_1|_Y$. 
We pick a smooth vector field $V\in \Gamma(TM_0|_Y)$ along $Y$ such that $|V|_g=1$ and $\langle V, N_0\rangle>0$ where $N_0$ is the inward pointing normal vector field of $Y_0\subset M_0$.
{\color{black}
We consider the map \begin{center} $(x,t)\in Y\times \R\mapsto  T(x,t)= \begin{cases}  \exp^{M_0}(t V(x)) \in M_0 & \mbox{ if } t\geq 0,\\
\exp^{M_1}(tV(x))\in M_1& \mbox{ if } t\leq 0
\end{cases}$. \end{center}}
\noindent
Since $Y_0\simeq Y_1=:Y$ is smooth and compact, and $DT_{(x,0)}=(\mbox{id}_{TY}, \mbox{id}_\R)$, there exists $\delta>0$ such that $T|_{Y_0\times (0, \delta)}: Y_0\times (0, \delta) \rightarrow U_0$ as well as $T|_{Y_1\times (-\delta,0)}: Y_1\times (- \delta, 0)\rightarrow U_1$  are  diffeomorphisms.  On $U_0$ we define a smooth  map $f_0$ via $P_2 \circ (T|_{Y_0\times (0,\delta)})^{-1}(p)=f_0(p)$. Hence $f_0\circ T(x,t)= t$ and $f_0$ is $1$-Lipschitz. 
Similarly,
 we define $f_1(p)= P_2\circ (T|_{Y\times (-\delta,0)})^{-1}(p)$. The function 
$$f(p)= \begin{cases} f_0(p) & p\in U_0\subset M_0\\
0& p\in Y\\
f_1(p) & p\in U_1\subset M_1
\end{cases}.
$$
is smooth if restricted to  $M_0$ or $M_1$, and $1$-Lipschitz on $U_0\cup Y\cup U_1=:U$. The gradient flow curves of $f$ on $B_\delta(Y)$ are 
$$
t\in (-\alpha_x, \omega_x) \rightarrow M, \ \gamma_x(t)=T(x,t)
$$ 
and $\gamma_x(t)$ are geodesics in $M_0\cup_{\mathcal I} M_1$. 

We can assume that $f$ is defined everywhere on $M_0\cup_{\mathcal I} M_1$ and $f$ induces again a disintegration of $\m$ according to Subsection \ref{subsec:1Dlocalisation}.

As before the map $T$  also provides  a disintegration that is strongly consistent w.r.t. $\{\gamma_x\}_{x\in Y}$, and hence
$$\m_x =h_x(t) dt= \frac{1}{\lambda_x} \det DT_{(x,t)} \Phi \circ T(x,r) dr. $$
Since $t\mapsto h_x(t)^\frac{1}{N}$ is semi-convex (Theorem \ref{thm:cavmil}), we have $\frac{d}{dt^+}_{t=0}  h_x(t) + \frac{d}{dt^-}_{t=0}  h_x(t)\geq 0$ on $Y_0\simeq Y_1$.
\medskip\\
{\bf 3.} In the following we fix a point $x_0= :x$ in  $Y$. 

The goal in the remaining steps will be to  compute $\frac{d}{dt^+}\big|_{t=0}  h_x(t)$ (and $\frac{d}{dt^-}\big|_{t=0} h_x(t)$) for a special choice of  $V$.
\smallskip

For this we first observe  that
$$
\det DT_{(x,t)}={ \det DT^t(x) }\beta_x(t)
$$
where $x\in Y\mapsto T^t(x)= \exp_x(tV(x))\in T(Y\times\{0\})$ and $$\beta_x(t)= \langle N(x,t), c_{V(x)}'(t)\rangle$$ with $c_{V(x)}(t)= \gamma_x(t)= \exp_x(tV(x))$ and $N(x,t)\in (\Im (DT^t(x)))^\perp$ with $|N(x,t)|=1$.  Since $T$ is smooth for $\delta$ sufficiently small, also $N(x,t)$ is a smooth vector field.

{\color{black}
Hence 
\begin{align*}\frac{d}{dt^{\pm}}\Big|_{t=0}h_x(t)=&\left( \frac{d}{dt^{\pm}}\Big|_{t=0} \det DT^t(x) \right) \beta_x(t) \Phi\circ T(x,t) \\
&+ \det DT^t(x) \left(\frac{d}{dt^{\pm}} \Big|_{t=0}\beta_x(t) \right)\Phi\circ T(x,t) 
\\
&+ \det DT^t(x) \beta_x(t)\left( \frac{d}{dt^{\pm}} \Big|_{t=0}\Phi\circ T(x,t)\right).\end{align*}
In the following  we will drop the superscript $\pm$ of $t^\pm$ if it is not relevant.}
\smallskip
\\
{\bf 3.1.}
We set  $DT^t(x)=:A_x(t)$ and $\det A_x(t)=:y_x(t)$. Differentiating $y_x(t)$ at $t=0$ yields 
$$
y'_x(0)= \mbox{tr}^{T_xY}\frac{d}{dt}\Big|_{t=0} A_x(t). 
$$
{\color{black} Hence, we compute $\frac{d}{dt}\Big|_{t=0} A_x(t)$.}

We choose an orthonormal basis $e_i\in T_xY$, $i=1, \dots, n-1$, w.r.t. $g_0|_Y$ such that ${\bf P}^\top V(x)=\hat ae_1$ for $\hat a\in \R\backslash\{0\}$.  ${\bf P}^\top: TM_0|_Y\simeq TM_1|_Y\rightarrow TY$ is the tangential projection. 

Then $A_x(t)e_i= DT^t(x) e_i=J_i(t)$ is a Jacobi field that satifies $$J'_i(0)= \nabla_{e_i}V\circ T(x,0)= \nabla_{e_i} V_x.$$ In particular, note that $N(x,t)\perp J_i(t)$ $\forall i=1, \dots, n-1$. 

We fix an orthonormal frame $(E_i)_{i=1, \dots, n-1}$ on $Y$ such that $e_i=E_i(x)$ for $i=1, \dots, n-1$, and  $\frac{{\bf P}^\top V}{|{\bf P}^\top V|}= E_1$. 

{\color{black}
Moreover, we choose $E_1$ (and therefore $V$) such that  $\nabla_{e_i} E_1|_x=0$, $i=1, \dots, n-1$. We will see at the end of the proof that this is no restriction for our purpose.}

There exist smooth functions $a, b\in C^\infty(Y)$ such that 
$$ V=
aE_1  + b N_0 \mbox{ on } Y_0 \mbox{ with } a(x)= \hat a \mbox{ and } b(x) =\hat b= \beta_x(0)$$
where $N_0$ is the unit normal vector field along $Y_0$.  $x$ is still our fixed point. 

We compute for $i=1, \dots, n-1$ that 
\begin{align*}\langle \nabla_{e_i} V, e_i\rangle=&
 \langle e_i, \nabla a|_x\rangle \langle e_1, e_i\rangle+ a(x) \langle \nabla_{e_i}E_1|_x, e_i\rangle\\
&\ \ \ + \langle e_i, \nabla b|_x\rangle\langle N_0, e_i\rangle+ b(x) \langle \nabla_{e_i} N_0, e_i\rangle.\end{align*}
The second term on the RHS is $0$ because $\nabla_{e_i} E_1|_x=0$ for $i=1, \dots, n-1$, while the third term on the RHS is $0$ because $N_0(x)\perp e_i$.  The last term on the RHS is $b$ times $\Pi_0|_x(e_i, e_i)$ where $\Pi_0$ is the second fundamental form of $Y_0$. For the first term on the RHS we notice that $\langle e_1, e_i\rangle =0$ if $i\neq 1$.
\smallskip

{
We still have to consider the first term on RHS for $i=1$.  }

Note that $a_+:=a= \langle E_1, V\rangle$ describes the outgoing angle of the vector field $V$  relative to $T_xY$. Similarly, $a_-:=-a=\langle -V, E_1\rangle$ describes  the incoming angle of $V$.  
Hence, $0=\langle e_1, \nabla (a_+ + a_-)\rangle =  \langle e_1, \nabla a_+\rangle + \langle e_1, \nabla a_-\rangle.$
We conclude that 
\begin{align*}\frac{d}{dt^+} y_x(t)|_{t=0}+ \frac{d}{dt^-}y_x(t)|_{t=0}&= b\sum_{i=1}^{n-1} \left(\Pi_0(E_i, E_i) + \Pi_1(E_i, E_i)\right).
\end{align*}
\noindent
{\bf 3.2.} Next we compute the left/right derivative of  $\beta_x(t)$ at $t=0$. We get
$$b_x'(0)= \langle \nabla_{V(x)} N(x,t)|_{(x,0)}, V(x)\rangle.$$
We extend  $E_1$ to vectorfield $E_1$ on $U_0$ such that $[N, E_1]_{(x,0)}=0$.

We decompose  $V$ again  into $V=aE_1 + bN_0$ and write 
\begin{align*}b_x'(0)=& a^2(x) \Pi_0(e_1, e_1)+ a(x)b(x) \langle e_1, \nabla_{N_0(x)} N_0(x,t)|_{(x,0)}\rangle \\
&+ a(x)b(x) \langle N_0(x), \nabla_{E_1} N_0|_x\rangle + b^2(x) \langle N_0(x), \nabla_{N_0(x)} N(x,t)|_{(x,0)}\rangle.
\end{align*}
Using that $\nabla$ is a Riemannian connection that is symmetric we see that the last 3 terms on the RHS vanish. 
For instance, since $[N, E_1]_{(x,0)}=0$, one has
\begin{align*}
\langle E_1(x), \nabla_{N_0(x)} N(x,t)|_{(x,0)}\rangle= - \langle \nabla_{N_0(x)} E_1|_x, N_0(x)\rangle= - \langle \nabla_{e_1} N_0|_x, N_0(x)\rangle.
\end{align*}
The last term is the normal component of the covariant derivative of the unit  vector field $N_0$ in a directon tangent to $Y$ and hence $0$. 
\smallskip

We obtain 
$$\left(\frac{d}{dt^+} \beta_x+ \frac{d}{d^-} \beta_x\right)\Big|_{t=0}= \hat a^2 \left( \Pi_0(e_1, e_1)+ \Pi_1(e_1,e_1)\right).$$
\smallskip
\noindent
{\bf 3.3} We also note that  
\begin{align*}\frac{d}{dt^{+}} \Phi\circ T(x,t)\Big|_{t=0}&= \langle \nabla \Phi(x), V(x)\rangle\\
&=  \langle \nabla \Phi (x), \hat a e_1+ b N_0(x))\rangle= \langle \nabla \Phi(x), e_1\rangle + b \langle \nabla \Phi(x), N_0(x)\rangle
\end{align*}
Since $\frac{d}{dt^{-}} \Phi\circ T(x,t)\Big|_{t=0}= -\langle \nabla \Phi(x), V(x)\rangle$, it follows 
$$\left(\frac{d}{dt^+} \Phi\circ T(x,t) + \frac{d}{dt^-} \Phi\circ T(x,t)\right)\Big|_{t=0}=b\left( \langle \nabla \Phi_0|_x, N_0(x)\rangle + \langle \nabla \Phi_1|_x, N_1(x)\rangle\right).$$
\noindent 
{\bf 4.} We can summarize our computations as follows
\begin{align*}
&\frac{d}{dt^+}_{t=0}  h_x(t) + \frac{d}{dt^-}_{t=0}  h_x(t)\\
&\ \ \ =\hat b^2\left(\sum_{i=1}^{n-1} \left(\Pi_0(E_i, E_i) + \Pi_1(E_i, E_i)\right)\right) \Phi(x)\\
&\ \ \ \ \ \ \ +\hat a^2 \left( \Pi_0(e_1, e_1)+ \Pi_1(e_1,e_1)\right) \Phi(x)\\
&\ \ \ \ \ \ \ +\hat b^2\left( \langle \nabla \ln \Phi_0|_x, N_0(x)\rangle + \langle \nabla\ln  \Phi_1|_x, N_1(x)\rangle\right)\Phi(x)
\end{align*}
Now we fix $v\in T_xY$ and let $\epsilon>0$ be arbitrary. We can pick $V$ as before such that ${\bf P}^\top V(x)=v$ and such that $|\langle V(x), N_0(x)\rangle |= |\hat b|< \epsilon $.   It is clear that we can choose $V$ and vector fields $E_1, \dots, E_{n-1}$ along $Y$ with the properties as before (for instance ${\bf P}^\top V(x)= \hat a E_1(x)$). 

Then our computation and the condition $CD(K,N)$ for $M_0\cup_{\mathcal I}M_1$ imply
$$\Pi_0(v,v)+ \Pi_1(v,v)=\Pi(v,v)\geq 0.$$
This finishes the proof. 
\end{proof}
\medskip
\noindent
{\bf Data Availability:} Data sharing not applicable to this article as no datasets were generated or analysed during
the current study.
\bibliography{new}
\bibliographystyle{amsalpha}
\end{document}